\documentclass[11pt, reqno]{amsart}

\usepackage{amssymb}
\usepackage[shortlabels]{enumitem}
\usepackage{etoolbox}
\usepackage{eucal}
\usepackage[margin=1in]{geometry}
\usepackage[colorlinks=true, allcolors=blue]{hyperref}
\usepackage[dvipsnames]{xcolor}

\numberwithin{equation}{section}
\allowdisplaybreaks

\newtheorem{corollary}{Corollary}[section]
\newtheorem{lemma}[corollary]{Lemma}
\newtheorem{proposition}[corollary]{Proposition}
\newtheorem{theorem}[corollary]{Theorem}

\theoremstyle{definition}
\newtheorem{definition}[corollary]{Definition}

\newtheorem{remark}[corollary]{Remark}

\renewcommand{\bar}[1]{\overline{#1}}
\renewcommand{\epsilon}{\varepsilon}
\newcommand{\C}{\mathbb{C}}
\newcommand{\cF}{\mathcal{F}}
\DeclareMathOperator{\ch}{ch}
\newcommand{\fsl}{\mathfrak{sl}}
\newcommand{\g}{\mathfrak{g}}
\newcommand{\gl}{\mathfrak{gl}}
\newcommand{\h}{\mathfrak{h}}
\newcommand{\ind}[2]{\textup{ind}_{#1}^{#2}}
\newcommand{\lie}[1]{\mathfrak{#1}}

\newcommand{\n}{\mathfrak{n}}

\newcommand{\U}{\textup{\bf U}}
\newcommand{\vspan}{\textup{span}_\C}
\DeclareMathOperator{\wt}{wt}
\newcommand{\z}{\mathfrak{z}}
\newcommand{\Z}{\mathbb{Z}}

\begin{document}

\title{Local Weyl modules and fusion products for the current superalgebra $\lie{sl}(1|2)[t]$}

\author{Matheus Brito}
\address{Mathematics Department\\
         Federal University of Paran\'a\\
         Curitiba, Paran\'a, Brazil, 81.531-980}
\email{mbrito@ufpr.br}

\author{Lucas Calixto}
\address{Mathematics Department\\
         Federal University of Minas Gerais\\
		   Belo Horizonte, Minas Gerais, Brazil, 30.123-970}
\email{lhcalixto@ufmg.br}

\author{Tiago Macedo}
\address{Departamento de Ci\^encia e Tecnologia\\
         Universidade Federal de S\~ao Paulo\\
         S\~ao Jos\'e dos Campos, S\~ao Paulo, Brazil, 12.247-014}
\email{tmacedo@unifesp.br}

\begin{abstract}
We study a class of modules, called Chari-Venkatesh modules, for the current superalgebra $\lie{sl}(1|2)[t]$. This class contains  other important modules, such as graded local Weyl, truncated local Weyl and Demazure-type modules. We prove that Chari-Venkatesh modules can be realized as fusion products of generalized Kac modules. In particular, this proves Feigin and Loktev's conjecture, that fusion products are independent of their fusion parameters, in the case where the fusion factors are generalized Kac modules. As an application of our results, we obtain bases, dimension and character formulas for Chari-Venkatesh modules.
\end{abstract}

\maketitle

\tableofcontents

\section*{Introduction}

Representation theory of Lie (super)algebras has been extensively researched for about a century, in particular, due to its relationship with other scientific fields, such as quantum mechanics.  In this paper, we will study certain classes of finite-dimensional representations for current superalgebras.  Our main goal is to provide a better description of modules known as \emph{local Weyl}.  In order to do that, we make use of \emph{fusion products} and \emph{Chari-Venkatesh modules}.

Local Weyl modules were first defined in \cite{CP01} as modules for the loop algebra associated to a semisimple finite-dimensional Lie algebra, in analogy with Weyl modules for reductive algebraic groups over fields of positive characteristics.  In fact, Chari and Pressley defined local Weyl modules via generators and relations, proved that they are universal highest-weight modules, and conjectured that they are in fact the classical limit of certain irreducible modules for the corresponding quantum affine algebra.  Their conjecture was proven in \cite[Theorem~5]{CP01} for $\g = \lie{sl}(2)$, in \cite[Theorem~1.5.1]{CL06} for $\g$ of type $\mathsf A$, in \cite[Corollary~2(1)]{FL07} for $\g$ simply laced, and in \cite[Corollary~9.5(i)]{Naoi12} for the remaining cases.

Simultaneously, the study of local Weyl modules was developed in different directions, as it became clear that they are related to several other interesting objects.  For instance: in \cite{FL04}, the definition of local Weyl modules was generalized to map algebras, and they were related to Catalan numbers; in \cite{CFK10}, they were fit into a categorical viewpoint; in \cite{KN12}, they were related to quiver varieties; and in \cite{BHLW17, SVV17}, they were related to decategorification of quantum groups.

In \cite{CLS19}, local Weyl modules for map superalgebras were first defined.  In that paper, the authors defined local Weyl modules for map superalgebras $\g \otimes_\C A$ where $\g$ is a basic classical Lie superalgebra and $A$ is a finitely-generated unital $\C$-algebra, and proved some of the fundamental properties of these modules. In \cite{BCM19}, the authors extended the definition of local Weyl modules to map superalgebras $\g \otimes_\C A$ with $\g$ being any finite-dimensional simple Lie superalgebra non-isomorphic to $\lie{q}(n)$, and fit these local Weyl modules into a categorical framework similar to the one in \cite{CFK10}. In \cite{FM17} and \cite{Kus18}, the authors focused on the study of local Weyl modules for the current superalgebra $\lie{osp}(1|2)[t]$.  In \cite{FM17}, the authors computed the dimension and characters of such modules and related them to non-symmetric Macdonald polynomials; and in \cite{Kus18}, the author realized local Weyl modules as certain fusion products, provided generators and relations for fusion products, and studied Demazure-type and truncated Weyl modules.

Our goal in this paper is to deepen the understanding of local Weyl modules and fusion products, by studying the case of the current superalgebra $\lie{sl}(1|2)[t]$.  Notice that the Lie superalgebra $\lie{sl}(1|2)$ is the most fundamental simple Lie superalgebra for which the category of finite-dimensional modules is not semisimple.  The category of finite-dimensional $\lie{osp}(1|2)$-modules, on the other hand, is semisimple, and thus the representation theory of $\lie{osp}(1|2)[t]$ behaves in many ways like the representation theory of current Lie algebras.  Further, notice that, since the even part $\lie{sl}(1|2)_{\bar0}$ is isomorphic to $\lie{gl}(2)$, we are able to make very explicit calculations.  We expect that our calculations and techniques enable one to treat the general case of current superalgebras in the future.

An important feature of the representation theory of map superalgebras is their dependence on a choice of Borel subalgebra.  This is due to the fact that, for finite-dimensional simple Lie superalgebras, not all Borel subalgebras are conjugate under the action of the corresponding Weyl group.  In fact, local Weyl modules are universal objects in the category of finite-dimensional $\lie{sl}(1|2)[t]$-modules if and only if one considers certain triangular decompositions, whose highest root is even (see \cite[Corollary~8.16(b)]{BCM19}).  Moreover, for any other choice of triangular decomposition, there exists no objects in the category of finite-dimensional $\lie{sl}(1|2)[t]$-modules satisfying those universal properties (see \cite[Remark~8.15]{BCM19}).  Since, up to conjugation under the action of the Weyl group, there exists a unique triangular decomposition whose highest root is even, we lose no generality in fixing one.  The Borel subalgebra associated to this fixed triangular decomposition is denoted by $\lie{b}_{\scriptscriptstyle(2)}$ and explicitly defined in Section~\ref{ss:good.borel}.

In Section~\ref{s:kac}, we describe the structure of generalized Kac modules associated to the Borel subalgebra $\lie b_{\scriptscriptstyle (2)}$.  In fact, we describe when these generalized Kac modules are irreducible (Proposition~\ref{prop:typ.cond}), when they are non-zero (Proposition~\ref{prop:dim.k2.sl12}), and their $\g_{\bar0}$-module decompositions (Proposition~\ref{prop:kac.g0-mods}).  This description is important due to the fact that local Weyl modules associated to $\lie{b}_{\scriptscriptstyle(2)}$ are finite-dimensional \cite[Theorem~8.7]{BCM19} (in contrast to local Weyl modules associated to $\lie{b}_{\scriptscriptstyle(1)}$ or $\lie{b}_{\scriptscriptstyle(3)}$).  Notice that these results generalize Kac's description of Kac modules associated to distinguished Borel subalgebras \cite[\S2.2]{Kac78}.  In fact, the main tool used to describe the structure of generalized Kac modules associated to $\lie b_{\scriptscriptstyle(2)}$ is Lemma~\ref{lem:kac.iso}, which works for any finite-dimensional simple Lie superalgebra (not just $\lie{sl}(1|2))$.

After that, we fix $\g = \lie{sl}(1|2)$, the Borel subalgebra $\lie b_{\scriptscriptstyle(2)}$, and proceed to studying the associated local Weyl modules in Section~\ref{s:weyl}.  In Proposition~\ref{prop:Weyl_basis}, we construct a generating set for graded local Weyl modules.  This set is proved to be a basis in Theorem~\ref{thm:weyl.dim}, using fusion products.  An analogous result had been obtained for graded local Weyl $\lie{osp}(1|2)[t]$-modules in \cite[Corollary~1.8]{FM17}.  In their case, however, irreducible finite-dimensional $\lie{osp}(1|2)$-modules coincide with Kac modules (since the category of finite-dimensional $\lie{osp}(1|2)$-modules is semisimple).  In our case, that is, when $\g = \lie{sl}(1|2)$, we are forced to look at typical and atypical weights, and, in particular, we prove that local Weyl modules can be realized as fusion products of generalized Kac modules, which are not irreducible if the associated weight is atypical.

One should notice that fusion products are also used in \cite{FM17, Kus18}, as well as an important tool used in \cite{CL06, FL07, Naoi12} to prove Chari and Pressley's conjecture.  These modules were first defined in \cite[Definition~1.7]{FL99} to be graded versions of tensor products of evaluation modules evaluated at different points of $\C$ (known as \emph{fusion parameters}).  It was conjectured in \cite[Conjecture~1.8(i)]{FL99} that the module structure of a fusion product is independent of the choice of the fusion parameters.  This conjecture has been proven in a few cases.  In Corollary~\ref{cor:iso.vcsi}\ref{cor:CV.indep}, we prove that fusion products of generalized Kac modules are independent of the fusion parameters.  In particular, this result proves Feigin and Loktev's conjecture for  irreducible finite-dimensional typical $\lie{sl}(1|2)$-modules.

Moreover, we provide generators and relations for the fusion product of generalized Kac modules in Section~\ref{s:v.csi}.  In fact, we define Chari-Venkatesh modules for $\lie{sl}(1|2)[t]$ in Definition~\ref{defn:vcsi}, and prove that they are isomorphic to certain fusion products of generalized Kac modules in Proposition~\ref{prop:vcsi.onto.kacs} and Theorem~\ref{thm:iso.vcsi}.  As consequences of these results, we obtain bases, dimension and character formulas for Chari-Venkatesh $\lie{sl}(1|2)[t]$-modules in Corollary~\ref{cor:iso.vcsi}.  These results generalize \cite[Theorem~5]{CV15} (where $\g$ is $\lie{sl}(2)$) and \cite[Corollary~4.8]{Kus18} (where $\g$ is $\lie{osp}(1|2)$).

Particular cases of Chari-Venkatesh modules are graded local Weyl modules (see Lemma~\ref{lem:vcsi.basc}\ref{lem:vcsi.basc.b}).  In particular, Corollary~\ref{cor:iso.vcsi} yields bases, dimension and character formulas for graded local Weyl modules.  But in Corollary~\ref{cor:weyl.dim}, we also construct bases, prove dimension and character formulas for a larger class of local Weyl modules.  A key ingredient in the proof of Corollary~\ref{cor:weyl.dim} is Lemma~\ref{lem:dim.wpsi}.  Unfortunately, the argument used in the proof of Lemma~\ref{lem:dim.wpsi} is not valid for all local Weyl $\lie{sl}(1|2)[t]$-modules.  However, if one replaces $\lie{sl}(1|2)$ by $\lie{osp}(1|2)$, a similar argument would prove Lemma~\ref{lem:dim.wpsi}, Theorem~\ref{thm:gr.wpsi.weyl} and Corollary~\ref{cor:weyl.dim} for \emph{all} local Weyl $\lie{osp}(1|2)[t]$-modules (not just the graded ones), generalizing \cite[Corollary~1.8]{FM17}.  The authors believe that these results are yet unknown for $\lie{osp}(1|2)[t]$, and even for current Lie algebras.

We finish the paper with applications of our results to two classes of Chari-Venkatesh modules: Demazure-type modules and truncated local Weyl modules.  In fact, we first describe them in terms of fusion products (or, equivalently, Chari-Venkatesh modules) in Propositions~\ref{prop:dem.vcsi} and \ref{prop:tr.weyl.iso.fusion}, and then compute bases, dimension and character formulas in Corollaries~\ref{cor:app.dem} and \ref{cor:app.tr.weyl}.  Truncated local Weyl modules are universal objects in the category of finite-dimensional modules for truncated current superalgebras (see Proposition~\ref{prop:tr.weyl.univ}), and Demazure-type modules are expected to be isomorphic to modules for a certain Borel subgroup of the affine supergroup associated to ${\rm SL}(1|2)$ generated by extremal weight vectors (see Proposition~\ref{prop:dem.lw}).

\subsection*{Notation}  Throughout this paper, we will denote by $\C$ the set of complex numbers, by $\Z$ the set of integer numbers, by $\Z_{\ge0}$ the set of non-negative integers, and by $\Z_{>0}$ the set of positive integers.  Moreover, all vector spaces, algebras and tensor products will be assumed to be taken over $\C$, unless it is explicitly said otherwise. For a given Lie superalgebra ${\lie a}$ we let $\U({\lie a})$ denote its universal enveloping superalgebra.

\section{Generalities on \texorpdfstring{$\mathfrak{sl}(1|2)$}{sl(1|2)}} 
\label{s:sl(1|2)}

The Lie superalgebra $\g = \lie{sl}(1|2)$ is a finite-dimensional simple Lie superalgebra whose elements are $3 \times 3$ complex matrices of the form
\[
\g = \left\{ 
\left.
A = \left(
\begin{array}{c|cc}
a & x & y \\
\hline
z & b & c \\
w & d & e
\end{array}
\right)
\right| {\rm str}(A) := a - (b+e)=0 \right\}.
\]
Its even and odd parts are given by
\[
\g_{\bar0} = 
\left\{
\left.
\left(
\begin{array}{c|cc}
a+d & 0 & 0 \\
\hline
0 & a & b \\
0 & c & d
\end{array}
\right)
\right| a, b, c, d \in \C
\right\}
\quad \textup{and} \quad
\g_{\bar1} = \g_{1} \oplus \g_{-1},
\]
where
\[
\g_1 = 
\left\{
\left.
\left(
\begin{array}{c|cc}
0 & x & y \\
\hline
0 & 0 & 0 \\
0 & 0 & 0
\end{array}
\right)
\right| x, y \in \C
\right\}
\quad \textup{and} \quad
\g_{-1} = 
\left\{
\left.
\left(
\begin{array}{c|cc}
0 & 0 & 0 \\
\hline
z & 0 & 0 \\
w & 0 & 0
\end{array}
\right)
\right| z, w \in \C
\right\}.
\]
And its Lie bracket is defined by linearly extending:
\[
[A, B] = AB - (-1)^{\bar i \, \bar j} BA
\qquad \textup{ for all } A \in \g_{\bar i},\ B \in \g_{\bar j},\ \bar i, \bar j \in \Z_2.
\]
Moreover, $\g = \g_{-1} \oplus \g_0 \oplus \g_{1}$, with $\g_0 = \g_{\bar0}$, is a $\Z$-grading of $\g$.

Notice that $\g_0\cong \gl(2)$ and $\g_0':=[\g_0,\g_0]\cong \fsl(2)$. Thus, $\g_0$ is a reductive (non-semisimple) Lie algebra, and we can choose a triangular decomposition $\g_0 = \n_0^- \oplus \h \oplus \n_0^+$ corresponding to the usual triangular decomposition of $\gl(2)$, namely:
\[
\n_0^- = \vspan \{y_2:=E_{3,2}\},
\quad
\n_0^+ =  \vspan \{x_2:= E_{2,3}\}
\quad \textup{and} \quad
\h = \h' \oplus \z,
\]
where
\[
\h' =  \vspan \{h_2:=E_{2,2} - E_{3,3} \}
\quad \textup{and} \quad
\z =  \vspan \{z:=2E_{1,1} + E_{2,2} + E_{3,3} \}.
\]
In fact, $\h$ is a Cartan subalgebra of $\g_0$ and $\z$ is its center.  Moreover, $\lie b_0 := \h \oplus \n_0^+$ is a Borel subalgebra of $\g_0$.

A Cartan subalgebra of $\g$ is defined to be a Cartan subalgebra of $\g_0$.  Fix the Cartan subalgebra $\h \subseteq \g$ above, a basis for $\h$ consisting of $\{h_1 := E_{1,1} + E_{2,2},\ h_2\}$, and define $h_3 := E_{1,1} + E_{3,3}$.

The adjoint action of $\h$ on $\g$ induces the root space decomposition $\g = \bigoplus_{\alpha \in \h^*} \g^\alpha$, where $\g^\alpha$ is defined to be $\{ x \in \g \mid [h,x] = \alpha(h)x, \textup{ for all $h \in \h$} \}$.  The set of roots $R := \{ \alpha \in \h^*\setminus \{0\} \mid \g^\alpha \neq 0 \}$, consists of
\[
\pm \alpha_1 = \pm (\delta - \varepsilon_1), \quad
\pm \alpha_2 = \pm (\varepsilon_1 - \varepsilon_2)
\quad \textup{and} \quad
\pm \alpha_3 = \pm (\delta - \varepsilon_2),
\]
where $\delta, \epsilon_1, \epsilon_2$ are the unique linear maps that satisfy
	\[
\delta \left(
\begin{array}{c|cc}
a & 0 & 0 \\
\hline
0 & b_1 & 0 \\
0 & 0 & b_2
\end{array}
\right) = a,
\quad 
\epsilon_1 \left(
\begin{array}{c|cc}
a & 0 & 0 \\
\hline
0 & b_1 & 0 \\
0 & 0 & b_2
\end{array}
\right) = b_1,
\quad
\epsilon_2 \left(
\begin{array}{c|cc}
a & 0 & 0 \\
\hline
0 & b_1 & 0 \\
0 & 0 & b_2
\end{array}
\right) = b_2.
	\]
In particular, $\g^{\alpha_i}=\C x_i$, $\g^{-\alpha_i}=\C y_i$, where $x_1:=E_{1,2},\ x_2:=E_{2,3},\ x_3:=E_{1,3}$, and $y_i:=x_i^t$ (the transpose of $x_i$).

Recall that $\g$ admits a non-degenerate even invariant super-symmetric bilinear form $(\cdot \mid \cdot)$ given by the super trace, that is, $(A \mid B)={\rm str}(AB)$ for all $A,B\in \g$. Such form is non-degenerate when restricted to $\h$ and hence it naturally induces a non-degenerate form on $\h^*$.

For each $\lambda \in \h^*$, denote $\lambda(h_i)$ by $\lambda_i$, and identify $\lambda$ with $(\lambda_1, \lambda_2) \in \C^2$. (Since $\{h_1, h_2\}$ is a basis for $\h$, every $\lambda \in \h^*$ is uniquely determined by the pair $(\lambda_1, \lambda_2) \in \C^2$.)  Alternatively, we sometimes identify $\lambda$ with $(\kappa; a, b)\in \C^3$, where $\lambda = \kappa\delta + a\varepsilon_1 + b\varepsilon_2$.  For each $\lambda = (\lambda_1, \lambda_2) \in \h^*$, let $L_{\gl}(\lambda)$ denote the irreducible $\g_0$-module with ${\lie b}_0$-highest weight $\lambda$, and $L_{\fsl}(\lambda_2)$ denote the irreducible $\g_0'$-module with $({\lie b}_0 \cap \g_0')$-highest weight $\lambda \arrowvert_{\h'} \in (\h')^*$.

\subsection{Distinguished sets of roots} \label{ss:good.borel}

There exist a few choices of sets of simple and positive roots for $\g$ (given that $\h$ is fixed).   A set of simple roots is called \emph{distinguished} if it has only one odd root.

Fix the following set of simple roots: $\Pi_{\scriptscriptstyle (1)} = \{ \alpha_1, \alpha_2 \} \subseteq R$.  Since $\alpha_3 = \alpha_1 + \alpha_2$, the set of positive roots associated to $\Pi_{\scriptscriptstyle (1)}$ is $R_{\scriptscriptstyle (1)}^+ = \{ \alpha_1, \alpha_2, \alpha_3 \}$.  Also notice that $\alpha_1$ and $\alpha_3$ are isotropic odd roots, i.e., $\alpha_1 (h_1) = \alpha_3 (h_3) = 0$ and $\g^{\alpha_1}, \g^{\alpha_3} \subseteq \g_1$. Associated to this choice of simple roots, we have the triangular decomposition $\g = \n_{\scriptscriptstyle (1)}^- \oplus \h \oplus \n_{\scriptscriptstyle (1)}^+$, where $\n_{\scriptscriptstyle (1)}^\pm := \bigoplus_{\alpha \in R_{\scriptscriptstyle (1)}^+} \g^{\pm \alpha} = \g^{\pm \alpha_1} \oplus \g^{\pm \alpha_2} \oplus \g^{\pm \alpha_3}$.  Denote by $\lie b_{\scriptscriptstyle (1)}$ the associated (positive) Borel subalgebra $\h \oplus \n_{\scriptscriptstyle (1)}^+$.

There is another important choice of sets simple and positive roots for $\g$ that satisfies the following condition: the lowest root of $\g$ with respect to this set of positive roots is even. Here, a root $\alpha\in R$ is called lowest if $\alpha -\beta\notin R$ for any $\beta\in R^+$. Choose the set of simple roots to be $\Pi_{\scriptscriptstyle (2)} = \{ -\alpha_1, \alpha_3 \}$.  Since $\alpha_2 = \alpha_3 - \alpha_1$, the set of positive roots associated to $\Pi_{\scriptscriptstyle (2)}$ is $R_{\scriptscriptstyle (2)}^+ = \{ -\alpha_1, \alpha_2, \alpha_3 \}$. Notice that the lowest root of $\g$ associated to $R_{\scriptscriptstyle (2)}^+$ is $-\alpha_2$.  Associated to this choice of simple roots, we have the triangular decomposition $\g = \n_{\scriptscriptstyle (2)}^- \oplus \h\oplus \n_{\scriptscriptstyle (2)}^+$, where $\n_{\scriptscriptstyle (2)}^\pm := \bigoplus_{\alpha \in R_{\scriptscriptstyle (2)}^+} \g^{\pm \alpha} = \g^{\mp \alpha_1} \oplus \g^{\pm \alpha_2} \oplus \g^{\pm \alpha_3}$.  Denote by $\lie b_{\scriptscriptstyle (2)}$ the associated (positive) Borel subalgebra $\h \oplus \n_{\scriptscriptstyle (2)}^+$.

\subsection{Odd reflections} \label{ss:odd.refs}

The set of simple roots $\Pi_{\scriptscriptstyle (2)}$ can be obtained from $\Pi_{\scriptscriptstyle (1)}$ by applying an odd reflection.  In this subsection we describe this general construction.

Let $R$ be a set of roots.  Given a choice of simple roots $\Pi \subset R$, to each odd isotropic simple root $\alpha \in \Pi$, one can define an odd reflection $r_{\alpha} \colon \Pi \to R$ by
\[
r_{\alpha} (\beta) =
\begin{cases}
-\alpha, & \textup{if } \alpha=\beta \\
\beta, & \textup{if $\alpha\neq \beta$ and $\alpha (h_\beta)=\beta(h_\alpha) = 0$} \\
\beta + \alpha, & \textup{if either $\alpha (h_\beta)\neq 0$ or $\beta(h_\alpha) \ne 0$}.
\end{cases}
\]
For instance, if we choose $\Pi = \Pi_{\scriptscriptstyle (1)}$ and $\alpha = \alpha_1$, we obtain $r_{\alpha_1} (\Pi_{\scriptscriptstyle (1)}) = \Pi_{\scriptscriptstyle (2)}$.

Now, if we choose $\Pi = \Pi_{\scriptscriptstyle (2)}$ and $\alpha = \alpha_3$, then $\Pi_{\scriptscriptstyle (3)} = r_{\alpha_3} (\Pi_{\scriptscriptstyle (2)}) = \{ \alpha_2, -\alpha_3 \}$.  Since $\alpha_2 - \alpha_3 = - \alpha_1$, the set of positive roots associated to $\Pi_{\scriptscriptstyle (3)}$ is $R^+_3 = \{ - \alpha_1, \alpha_2, -\alpha_3 \}$.  Thus, the corresponding triangular decomposition $\g = \n^-_{\scriptscriptstyle (3)} \oplus \h \oplus \n^+_{\scriptscriptstyle (3)}$ is such that $\n^\pm_{\scriptscriptstyle (3)} = \n_0^\pm \oplus \g_{\mp 1}$.  Denote by $\lie b_{\scriptscriptstyle (3)}$ the Borel subalgebra associated to this choice of simple roots, $\lie b_{\scriptscriptstyle (3)} = \h \oplus \n_{\scriptscriptstyle (3)}^+$.  Notice that $\Pi_{\scriptscriptstyle (3)}$ is also a distinguished set of simple roots, since its only odd root is $-\alpha_3$.

It can be shown that all possible choices of simple roots for $\g$ (given that $\h$ is fixed) are $\pm \Pi_{\scriptscriptstyle (1)}$, $\pm\Pi_{\scriptscriptstyle (2)}$ and $\pm\Pi_{\scriptscriptstyle (3)}$.  Notice that, in the particular cases where $\Pi \in \{\Pi_{\scriptscriptstyle(1)}, \Pi_{\scriptscriptstyle(2)}, \Pi_{\scriptscriptstyle(3)}\}$, the corresponding Borel subalgebras have the same even part, namely, ${\lie b}_0$.

\section{Generalized Kac modules} \label{s:kac}

Let $\g = \lie{sl}(1|2)$ and $\h$ be the Cartan subalgebra of the reductive Lie algebra $\g_0$ chosen in Section~\ref{s:sl(1|2)}. A $\g$-module $V$ is a $\Z_2$-graded vector space $V=V_0\oplus V_1$ such that $\g_iV_j\subseteq V_{i+j}$ for all $i,j\in \Z_2$. A homomorphism between $\g$-modules $V,W$ is a linear map $f:V\to W$ such that $f(xv)=xf(v)$ for all $x\in \g$ and $v\in V$. Notice that we do not ask for homomorphisms between $\g$-modules to preserve parity. Recall that every irreducible finite-dimensional $\g$-module is a highest-weight module (with respect to any choice of Borel subalgebra), with a highest weight $\lambda$ for some $\lambda \in \h^*$.  Given a Borel subalgebra $\lie b \subseteq \g$ and $\lambda \in \h^*$, let $V_{\lie b} (\lambda)$ denote the irreducible $\g$-module of highest weight $\lambda$, and define
\[
P^+_\lie b = \{ \lambda \in \h^* \mid V_\lie b (\lambda) \textup{ is finite dimensional} \}.
\]

Notice that $P^+_\lie b$ depends on the choice of $\lie b$.  For instance, we have that $V_{\lie b_{\scriptscriptstyle (1)}}(\delta)$ is the natural 3-dimensional $\g$-module, and hence $\delta \in P^+_{\lie b_{\scriptscriptstyle (1)}}$.  However, $\delta \not\in P^+_{\lie b_{\scriptscriptstyle (2)}}$.  In fact, there exists an explicit relation between irreducible $\g$-modules associated to different Borel subalgebras, given by \cite[Lemma~10.2]{Ser11}.  In particular, for $\lie b_{\scriptscriptstyle (1)}$ and $\lie b_{\scriptscriptstyle (2)}$, we have that
\begin{equation} \label{eq:iso.irreps}
V_{\lie b_{\scriptscriptstyle (1)}}(\lambda) \cong
\begin{cases}
V_{\lie b_{\scriptscriptstyle (2)}}(\lambda), & \textup{if $\lambda_1= 0$}, \\
V_{\lie b_{\scriptscriptstyle (2)}}(\lambda-\alpha_1), & \textup{if $\lambda_1\neq 0$},
\end{cases}
\quad \textup{ and } \quad
V_{\lie b_{\scriptscriptstyle (2)}}(\lambda) \cong
\begin{cases}
V_{\lie b_{\scriptscriptstyle (1)}}(\lambda), & \textup{if $\lambda_1= 0$}, \\
V_{\lie b_{\scriptscriptstyle (1)}}(\lambda+\alpha_1), & \textup{if $\lambda_1\neq 0$}.
\end{cases}
\end{equation}
Since $\delta (h_1) \ne 0$, by \eqref{eq:iso.irreps}, there is an isomorphism of $\g$-modules $V_{\lie b_{\scriptscriptstyle (2)}} (\delta) \cong V_{\lie b_{\scriptscriptstyle (1)}} (\delta+\alpha_1)$.  Now, notice that a highest-weight generator of $V_{\lie b_{\scriptscriptstyle (1)}} (\delta+\alpha_1)$ also generates a highest-weight submodule for the Lie subalgebra $\g'_0 \cong \lie{sl}(2)$ of highest weight $(\delta + \alpha_1)\arrowvert_{\h'}$.  Since $(\delta + \alpha_1)(h_2) = -1$, we see that this submodule is not finite dimensional.  Thus $V_{\lie b_{\scriptscriptstyle (2)}} (\delta) \cong V_{\lie b_{\scriptscriptstyle (1)}} (\delta+\alpha_1)$ is not finite dimensional.

We now recall the definition of generalized Kac modules, given in \cite[Definition~4.1]{BCM19}, for the particular case where $\g = \lie{sl}(1|2)$.  In order to do that, choose a triangular decomposition $\g = \lie n^- \oplus \h \oplus \lie n^+$ and, thus, a Borel subalgebra $\lie b = \h \oplus \n^+$.  

\begin{definition}
For each $\lambda \in P^+_\lie b$, define the \emph{generalized Kac module} associated to $\lambda$ to be the cyclic $\g$-module
	\[
K_{\lie b}(\lambda):=\U(\g)/\langle \lie n^+,\ h-\lambda(h),\ y_2^{\lambda_2+1}\rangle,
	\]
where the action of $\g$ on $K_{\lie b}(\lambda)$ is induced from the left multiplication of $\U(\g)$.  Denote by $k_\lambda$ the image of $1 \in \U(\g)$ onto the quotient. Notice that $k_\lambda$ is a homogeneous even vector, and that $K_{\lie b}(\lambda)$ is generated by $k_\lambda$ with defining relations
	\[
\lie n^+ k_\lambda=0, \qquad
h k_\lambda=\lambda(h)k_\lambda\quad \textup{for all $h\in \h$},
\qquad \textup{and} \qquad
y_2^{\lambda_2+1}k_\lambda=0.
	\]
\end{definition}

It was shown in \cite[Proposition~4.8]{BCM19} that the generalized Kac module $K_{\lie b}(\lambda)$ is a universal object in category of finite-dimensional $\g$-modules with $\lie b$-highest weight $\lambda\in \h^*$.  In particular, for each $\lambda \in P^+_{\lie b}$, we have that $V_{\lie b} (\lambda)$ is a quotient of $K_{\lie b} (\lambda)$. Notice that the $\Z_2$-grading on $V_{\lie b} (\lambda)$ is that induced from the $\Z_2$-grading of $K_{\lie b} (\lambda)$.

In the particular case where $\lie b$ is distinguished, the following result regarding the structure of Kac modules is known.

\begin{proposition}[{\cite[\S2.2]{Kac78}}] \label{prop:Kac_dist_properties}
Let $\lambda \in P^+_{\lie b}$, $\lie b \in \{\lie b_{\scriptscriptstyle(1)}, \lie b_{\scriptscriptstyle(3)}\}$.

\begin{enumerate}[(a), leftmargin=*]
\item \label{prop:Kac_dist_properties_1}
$K_{\lie b}(\lambda)$ is irreducible if and only if $\lambda$ is typical, that is,  
\begin{equation} \label{eq:typical.conditions}
( \lambda + \rho \mid \alpha_1 ) \ne 0 \quad \textup{and} \quad ( \lambda + \rho \mid \alpha_3 ) \ne 0,
\end{equation}
where $\rho = \varepsilon_1 - \delta$.

\item \label{prop:Kac_dist_properties_2}
There exist an isomorphism of $\g$-modules $K_{\lie b_{\scriptscriptstyle (1)}} (\lambda) \cong \ind{\lie g_0 \oplus \lie g_1}{\g} L_{\gl}(\lambda)$ and an isomorphism of $\g_0$-modules $K_{\lie b_{\scriptscriptstyle (1)}} (\lambda) \cong \Lambda(\g_{-1}) \otimes_\C L_{\gl}(\lambda)$, where $\Lambda(\g_{-1})$ is the Grassmann algebra associated to the $\g_0$-module $\g_{-1}$.  Similarly, there exist an isomorphism of $\g$-modules $K_{\lie b_{\scriptscriptstyle(3)}} (\lambda) \cong \ind{\lie g_0 \oplus \lie g_{-1}}{\g} L_{\gl}(\lambda)$ and an isomorphism of $\g_0$-modules $K_{\lie b_{\scriptscriptstyle(3)}} (\lambda) \cong \Lambda(\g_1) \otimes_\C L_{\gl}(\lambda)$, where $\Lambda(\g_1)$ is the Grassmann algebra associated to the $\g_0$-module $\g_1$.

\item \label{prop:Kac_dist_properties_3} $P_{\lie b_{\scriptscriptstyle (1)}}^+ = P_{\lie b_{\scriptscriptstyle (3)}}^+ = \{\lambda \in \h^* \mid \lambda_2 \in \Z_{\geq 0}\}$.
\end{enumerate}
\end{proposition}

For the rest of this section, we will prove analogs of the results of Proposition~\ref{prop:Kac_dist_properties} for the Borel subalgebra $\lie b_{\scriptscriptstyle(2)}$.  Our main goal is to describe the structure of generalized Kac modules associated to the Borel subalgebra $\lie b_{\scriptscriptstyle(2)}$, as this information will be used in Sections~\ref{s:weyl}--\ref{s:dem.tr.weyl}.

We begin with a lemma, which is an analog of \eqref{eq:iso.irreps} for generalized Kac modules, and will be used throughout the rest of this section.  Notice that its statement and proof are valid for any finite-dimensional simple Lie superalgebra; that is, for this lemma, $\g$ does not have to be $\lie{sl}(1|2)$.

\begin{lemma} \label{lem:kac.iso}
Let $\Pi$ be a set of simple roots for $\g$, let $\lie b \subseteq \g$ be the Borel subalgebra associated to $\Pi$, let $\lambda \in P^+_\lie b$, let $\alpha \in \Pi$ be an odd isotropic root, $\Pi'=\{ r_{\alpha} (\beta) \mid \beta \in \Pi \}$ and $\lie b' = \h \oplus {\n'}^+$ be the Borel subalgebra of $\g$ associated to $\Pi'$. If $\lambda (h_\alpha) \ne 0$ there is an isomorphism of $\g$-modules $K_{\lie b} (\lambda) \cong K_{\lie b'} (\lambda - \alpha)$.
\end{lemma}
\begin{proof}
This proof is similar to that of \cite[Lemma~10.2]{Ser11}. Let $k$ (resp. $k'$) denote a highest-weight generator of $K_{\lie b} (\lambda)$ (resp. $K_{\lie b'} (\lambda - \alpha)$).  Notice that, with respect to $\lie b'$, $y_\alpha k$ is a highest-weight generator of weight $\lambda - \alpha$ for $K_{\lie b} (\lambda)$.
Thus, by \cite[Propositions~4.4(a) and 4.8]{BCM19}, $K_{\lie b} (\lambda)$ and $K_{\lie b'} (\lambda - \alpha)$ are finite-dimensional $\g$-modules and there exists a uniquely defined homomorphism of $\g$-modules satisfying
\[
\phi \colon K_{\lie b'} (\lambda - \alpha) \to K_{\lie b} (\lambda), \quad
\phi (k') = y_\alpha k.
\]
Moreover, $\phi$ is surjective, since $y_\alpha k$ is also a generator of $K_{\lie b} (\lambda)$.

Similarly, we see that, with respect to $\lie b$, $x_\alpha k'$ is a highest-weight generator of weight $\lambda$ for $K_{\lie b'} (\lambda - \alpha)$.  Thus, there exists a uniquely defined (surjective) homomorphism of $\g$-modules satisfying
\[
\phi' \colon K_{\lie b} (\lambda) \to K_{\lie b'} (\lambda - \alpha), \quad
\phi' (k) = x_\alpha k'.
\]
The result  follows from the fact that $\frac{1}{\lambda (h_\alpha)} \phi' = \phi^{-1}$.
\end{proof}

In the next result, we use Lemma~\ref{lem:kac.iso} and Proposition~\ref{prop:Kac_dist_properties}\ref{prop:Kac_dist_properties_1} to prove an analog of Proposition~\ref{prop:Kac_dist_properties}\ref{prop:Kac_dist_properties_1} for the case where $\lie b = \lie b_{\scriptscriptstyle(2)}$.

\begin{proposition} \label{prop:typ.cond}
Let $\lambda \in P_{\lie b_{\scriptscriptstyle (2)}}^+$, and write it as $\lambda = (\kappa ; a, b) \in \C^3$, where $\lambda = \kappa \delta + a \varepsilon_1 + b \varepsilon_2$.  Then $K_{\lie b_{\scriptscriptstyle (2)}}(0)\cong \C$, and for $\lambda\neq 0$, $K_{\lie b_{\scriptscriptstyle (2)}} (\lambda)$ is irreducible if and only if $\kappa+a\neq 0$ and $\kappa+b\neq 0$.
\end{proposition}
\begin{proof}
The fact that $K_{\lie b_{\scriptscriptstyle (2)}}(0)\cong \C$ is obvious from the defining relations of generalized Kac modules, since $x_2,x_3,y_1\in \lie n_{(2)}^+$ (see Section \ref{ss:good.borel}), $x_1 = [x_3,y_2]$ and $y_3=[y_2,y_1]$.

Assume now that $\lambda\neq 0$. Recall that a weight $\nu \in P^+_{\lie b_{\scriptscriptstyle (1)}}$ is said to be typical when $K_{\lie b_{\scriptscriptstyle (1)}} (\nu)$ is irreducible.  By Proposition~\ref{prop:Kac_dist_properties}\ref{prop:Kac_dist_properties_1}, $\nu \in P^+_{\lie b_{\scriptscriptstyle (1)}}$ is typical if and only if conditions \eqref{eq:typical.conditions} are valid, where $\rho = - \alpha_1 = \varepsilon_1 - \delta$.

Since $\lambda\neq 0$ and $\{h_1,h_3\}$ is a basis of $\h$, then: either $\lambda_1 = \kappa+a\neq 0$, or $\lambda_3 = \kappa+b\neq 0$.  If $\kappa + a \ne 0$, then $\lambda_1 \ne 0$.  In this case, Lemma~\ref{lem:kac.iso} implies that $K_{\lie b_{\scriptscriptstyle (2)}} (\lambda) \cong_\g K_{\lie b_{\scriptscriptstyle (1)}} (\lambda + \alpha_1)$.  Since $\lambda + \alpha_1 + \rho = \lambda$, then \eqref{eq:typical.conditions} implies that $K_{\lie b_{\scriptscriptstyle (2)}} (\lambda)$ is irreducible if and only if $( \lambda \mid \alpha_1) = \kappa + a \ne 0$ and $(\lambda \mid \alpha_3) = \kappa + b \ne 0$.  The proof for the case $\kappa+b\neq 0$ is similar (replace ${\lie b_{\scriptscriptstyle (1)}}$ by ${\lie b}_{\scriptscriptstyle (3)}$).
\end{proof}

In the next result, we prove an analog of Proposition~\ref{prop:Kac_dist_properties}\ref{prop:Kac_dist_properties_3} and compute the dimension of generalized Kac modules for the Borel subalgebra $\lie b_{\scriptscriptstyle(2)}$.

\begin{proposition} \label{prop:dim.k2.sl12}
$P^+_{\lie b_{\scriptscriptstyle (2)}} = \{ \lambda \in \h^* \mid  \lambda_2 \in \Z_{\ge1} \}\cup \{0\}$, and $\dim K_{\lie b_{\scriptscriptstyle (2)}} (\lambda) = 4 \lambda_2$ for all $\lambda \in P^+_{\lie b_{\scriptscriptstyle(2)}} \setminus \{0\}$.
\end{proposition}
\begin{proof}
Since $K_{\lie b_{\scriptscriptstyle(2)}}(0) = \C$  we trivially have $0 \in P^+_{\lie b_{\scriptscriptstyle(2)}}$.

Now, assume that $\lambda \in \h^* \setminus \{ 0 \}$.  Since $\{h_1, h_3\}$ is a basis for $\h$, either $\lambda_1 \ne 0$ or $\lambda_3 \ne 0$.  By \eqref{eq:iso.irreps}, if $\lambda_1 \ne 0$, then $V_{\lie b_{\scriptscriptstyle (2)}}(\lambda) \cong_\g V_{\lie b_{\scriptscriptstyle (1)}} (\lambda+\alpha_1)$, and if $\lambda_3 \ne 0$, then $V_{\lie b_{\scriptscriptstyle (2)}}(\lambda) \cong_\g V_{\lie b_{\scriptscriptstyle (3)}} (\lambda-\alpha_3)$.  By Proposition~\ref{prop:Kac_dist_properties}\ref{prop:Kac_dist_properties_2}, $V_{\lie b_{\scriptscriptstyle (1)}} (\lambda+\alpha_1)$ is finite dimensional if and only if $(\lambda+\alpha_1)(h_2)\in \Z_{\geq 0}$, and $V_{\lie b_{\scriptscriptstyle (3)}} (\lambda-\alpha_3)$ is finite dimensional if and only if $(\lambda-\alpha_3)(h_2)\in \Z_{\geq 0}$. Thus, $V_{\lie b_{\scriptscriptstyle (2)}}(\lambda)$ is finite dimensional if and only if $\lambda_2 \in \Z_{\geq 1}$.  This proves the first statement.

Finally, assume that $\lambda \in P^+_{\lie b_{\scriptscriptstyle (2)}} \setminus \{0\}$.  By Lemma~\ref{lem:kac.iso}, if $\lambda_1 \ne 0$, then $K_{\lie b_{\scriptscriptstyle (2)}} (\lambda) \cong_\g K_{\lie b_{\scriptscriptstyle (1)}} (\lambda+\alpha_1)$, and if $\lambda_3 \ne 0$, then $K_{\lie b_{\scriptscriptstyle (2)}} (\lambda) \cong_\g K_{\lie b_{\scriptscriptstyle (3)}}(\lambda-\alpha_3)$.  By Proposition~\ref{prop:Kac_dist_properties}\ref{prop:Kac_dist_properties_3}, if $K_{\lie b_{\scriptscriptstyle (2)}} (\lambda) \cong_\g K_{\lie b_{\scriptscriptstyle (1)}} (\lambda+\alpha_1)$, then $\dim K_{\lie b_{\scriptscriptstyle (2)}} (\lambda) = 4 \dim L_{\gl}(\lambda + \alpha_1)$, and if $K_{\lie b_{\scriptscriptstyle (2)}} (\lambda) \cong_\g K_{\lie b_{\scriptscriptstyle (3)}}(\lambda-\alpha_3)$, then $\dim K_{\lie b_{\scriptscriptstyle (2)}} (\lambda) = 4 \dim L_{\gl}(\lambda - \alpha_3)$.  Since $\dim L_{\gl}(\lambda + \alpha_1) = \dim L_{\gl}(\lambda - \alpha_3) = \lambda_2$, the result follows.
\end{proof}

In the last result of this section we use Proposition~\ref{prop:Kac_dist_properties} and the usual representation theory of $\lie{sl}(2)$ to refine Proposition~\ref{prop:Kac_dist_properties}\ref{prop:Kac_dist_properties_2} and obtain decompositions of $K_{\lie b_{\scriptscriptstyle (1)}}(\lambda)$, $K_{\lie b_{\scriptscriptstyle (2)}}(\lambda)$, $K_{\lie b_{\scriptscriptstyle (3)}}(\lambda)$ as $\g_0$-modules.

\begin{proposition} \label{prop:kac.g0-mods}
Let $\lambda = (\lambda_1, \lambda_2) \in P^+_{\lie b}$, $\lie b \in \{\lie b_{\scriptscriptstyle (1)}, \lie b_{\scriptscriptstyle (2)}, \lie b_{\scriptscriptstyle (3)}\}$. We have the following isomorphisms of $\g_0$-modules:
\begin{enumerate}[(a), leftmargin=*]
\item \label{cor:kac.b1}
$K_{\lie b_{\scriptscriptstyle (1)}} (\lambda_1, \lambda_2) \cong L_{\gl}(\lambda_1, \lambda_2)\oplus L_{\gl}(\lambda_1 - 1, \lambda_2) \oplus L_{\gl}(\lambda_1, \lambda_2 + 1)\oplus L_{\gl}(\lambda_1 - 1, \lambda_2 -1)^{\oplus (1-\delta_{\lambda_2, 0})}$,

\item \label{cor:kac.b3}
$K_{\lie b_{\scriptscriptstyle (3)}} (\lambda_1, \lambda_2) \cong L_{\gl}(\lambda_1, \lambda_2)\oplus L_{\gl}(\lambda_1+1, \lambda_2) \oplus L_{\gl}(\lambda_1 + 1, \lambda_2+1)\oplus L_{\gl}(\lambda_1, \lambda_2-1)^{\oplus (1-\delta_{\lambda_2, 0})}$.
\end{enumerate}
In particular, if $\lambda\neq 0$, then the following isomorphism of $\g_0$-modules holds:
\begin{enumerate}[(c), leftmargin=*]
\item  \label{cor:kac.b2}
$K_{\lie b_{\scriptscriptstyle (2)}} (\lambda_1, \lambda_2) \cong L_{\gl}(\lambda_1, \lambda_2-1)\oplus L_{\gl}(\lambda_1-1, \lambda_2-1) \oplus L_{\gl}(\lambda_1, \lambda_2)\oplus L_{\gl}(\lambda_1-1, \lambda_2-2)^{\oplus (1-\delta_{\lambda_2, 1})}$.
\end{enumerate}
\end{proposition}
\begin{proof}
The proofs of parts~\ref{cor:kac.b1} and \ref{cor:kac.b3} follow from Proposition~\ref{prop:Kac_dist_properties}\ref{prop:Kac_dist_properties_2} and the usual fi\-ni\-te\--di\-mension\-al representation theory of $\lie{sl}(2)$.  If $\lambda\neq 0$, then either $K_{\lie b_{\scriptscriptstyle (2)}}(\lambda)\cong K_{\lie b_{\scriptscriptstyle (1)}}(\lambda+\alpha_1)$ or $K_{\lie b_{\scriptscriptstyle (2)}}(\lambda)\cong K_{\lie b_{\scriptscriptstyle (3)}}(\lambda-\alpha_3)$ (Lemma~\ref{lem:kac.iso}). Now, part~\ref{cor:kac.b2} follows from parts \ref{cor:kac.b1} and \ref{cor:kac.b3}, together with the facts that $(\lambda+\alpha_1)(h_1)=\lambda_1$, $(\lambda-\alpha_3)(h_1)=\lambda_1-1$ and $(\lambda+\alpha_1)(h_2)=(\lambda-\alpha_3)(h_2)=\lambda_2-1$.
\end{proof}

\section{Local Weyl modules} \label{s:weyl}

Throughout the rest of the paper we will assume that
	\[
\g := \lie{sl}(1|2) \ \text{ and } \ \lie b := \lie b_{\scriptscriptstyle (2)}.
	\]
Recall that, for a given $\lambda \in P_{\lie b}^+$, we denote $\lambda(h_1)$ by $\lambda_1$, $\lambda(h_2)$ by $\lambda_2$, and $\lambda$ by $(\lambda_1, \lambda_2) \in \C \times \Z_{\ge0}$.  Also, denote $K_{\lie b}(\lambda)$ simply by $K(\lambda)$ and $P^+_{\lie b} = \{ \lambda \in \h^* \mid  \lambda_2 \in \Z_{\ge1} \}\cup \{0\}$ simply by $P^+$.

Given any Lie superalgebra $\lie s$, denote by $\lie s[t]$ the \emph{current superalgebra} $\lie s \otimes \C[t]$ associated to $\lie s$.  For any element $x \in \lie s$ and $m \in \Z_{\ge0}$, the element $x \otimes t^m \in \lie s[t]$ will be denoted by $x(m)$.  By definition, for every $x, y \in \lie s$ and $m, n \in \Z_{\ge0}$, the Lie bracket in $\lie s[t]$ is given by $[x(m), y(n)] = [x,y](m+n)$.

We now recall the definition of local Weyl module given in \cite[Definition~4.1]{CLS19} for the particular case where $\g = \lie{sl}(1|2)$ and $A = \C[t]$.

\begin{definition} \label{defn:local.weyl}
Given $\psi \in \h[t]^*$ such that $\psi(h_2)\in \Z_{\geq 0}$, define the \emph{local Weyl module} $W(\psi)$ to be the $\g[t]$-module given as the quotient of $\U(\g[t])$ by the left ideal generated by
\begin{equation} \label{eq:defn.rels.wpsi}
\n^+[t], \qquad
H - \psi(H) \quad \textup{for all} \quad H \in \h[t], \qquad
y_2^{\psi(h_2)+1}.
\end{equation}
Notice that the local Weyl module is a highest-weight module of highest weight $\psi\arrowvert_{\h}$, generated by the image of $1 \in \U(\g[t])$ onto $W(\psi)$, which is a homogeneous even vector denoted by $w_\psi$.
\end{definition}

It has been proven in \cite[Corollary~8.16]{BCM19} that $W(\psi)$ is non-zero only if there exists a finite-codimensional ideal $I \subseteq \C[t]$ such that $\psi(\h \otimes I) = 0$.  Also, it has been proven in \cite[Theorem~8.13(b)]{BCM19} that $W(\psi)$ is finite dimensional if $\left. \psi \right|_\h \in P^+$.

Moreover, if $\psi \in \h[t]^*$ is such that $\psi(h \otimes t^k) = 0$ for all $h \in \h$, $k > 0$, then $\psi$ is uniquely determined by $\psi \arrowvert_\h \in P^+$.  On the other hand, every $\lambda \in P^+$ can be considered as an element of $\h[t]^*$ by defining $\lambda(h \otimes t^k) = \delta_{0,k} \lambda(h)$ for all $h \in \h$, $k \ge 0$.  When $\psi(h \otimes t^k) = 0$ for all $h \in \h$, $k > 0$, the left ideal of $\U(\g[t])$ generated by \eqref{eq:defn.rels.wpsi} is graded with respect to the degree of $t$.  Hence, the local Weyl module $W(\lambda)$ inherits a grading from $\U(\g[t])$, and we call it a graded local Weyl module.

\subsection{Graded local Weyl modules}

In what follows we focus on graded local Weyl modules as we intend to obtain a close relation between these modules and  fusion products of Kac modules.  Therefore, we fix $\lambda \in P^+$ and consider it as an element in $\h[t]^*$ as described above. In this subsection we construct a set of generators for $W(\lambda)$.

We begin by fixing some notation and proving a technical lemma, which will be used throughout the paper.  Given an element $u \in \U(\g[t])$ and $p \in \Z_{\ge0}$, define its $p$-th divided power to be $u^{(p)} := \frac{u^p}{p!}$.  Given $r, s \ge 0$, define 
\[
\mathbf y_2(r,s)
:= \sum_{\substack{b_0 + b_1 + \dotsb + b_s = r \\ b_1 + 2 b_2 + \dotsb + s b_s = s}} y_2^{(b_0)} y_2(1)^{(b_1)} \dotsm y_2(s)^{(b_s)}.
\]

\begin{lemma} \label{lem:comm.rels}
For all $a, a_1, \dotsc, a_i, b, b_1, \dotsc, b_k, c, c_1, \dotsc, c_\ell, r \in \Z_{\ge0}$, $s\in\Z_{\geq 1}$ and $h \in \lie h$, the following relations hold in $\U(\g[t])$:
\begin{align}
x_2(a)y_3(c_1)\cdots y_3(c_\ell)
={}&{} \left( \sum_{j=1}^\ell (-1)^{\ell-j} \prod_{i\neq j} y_3(c_i)y_1(c_j+a) \right) +
y_3(c_1)\cdots y_3(c_\ell)x_2(a); \label{eq:com.rel1} \\
h(a)y_3(c_1)\cdots y_3(c_\ell)
={}&{} \left(\alpha_3(h) \sum_{j=1}^\ell (-1)^{\ell-j+1} \prod_{i\neq j} y_3(c_i)y_3(c_j+a) \right) + y_3(c_1) \cdots y_3(c_\ell)h(a); \label{eq:com.rel2} \\
x_2(a)x_1(b_1)\cdots x_1(b_k)
={}&{} \left(\sum_{j=1}^k (-1)^{k-j+1} \prod_{i\neq j} x_1(b_i) x_3(b_j+a) \right) + x_1(b_1)\cdots x_1(b_k)x_2(a);  \label{eq:com.rel3} \\
h(a)x_1(b_1)\cdots x_1(b_k)
={}&{} \left(\sum_{j=1}^k (-1)^{k-j} \alpha_1(h) \prod_{i \ne j} x_1(b_i) x_1(b_j+a) \right) + x_1(b_1)\cdots x_1(b_k)h(a);  \label{eq:com.rel4} \\
x_3(b)y_3(c_1)\cdots y_3(c_\ell)
={}&{} \left( \sum_{j=1}^\ell (-1)^{j+1} \prod_{i\neq j} y_3(c_i)h_3(c_j+b) \right) +
(-1)^\ell y_3(c_1)\cdots y_3(c_\ell)x_3(b); \label{eq:com.rel5} \\
y_1(c) y_2(a_1) \dotsm y_2(a_j)
={}&{} - \left(\sum_{i=1}^j \prod_{p \ne i} y_2(a_p) y_3(a_i+c) \right) + y_2(a_1) \dotsm y_2(a_j) y_1(c); \label{eq:com.rel6} \\
x_3(b) y_2(a_1) \dotsm y_2(a_j)
={}&{} \left(\sum_{i=1}^j \prod_{p \ne i} y_2(a_p) x_1(a_i+b) \right) + y_2(a_1) \dotsm y_2(a_j) x_3(b); \label{eq:com.rel7} \\
x_2(1)^{(s)} y_2^{(r+s)} {}&{} - (-1)^s\mathbf y_2(r,s)
\in  \U(\g[t])(\n^+[t]\oplus t\h[t]); \label{eq:garland} \\
y_1(c) \mathbf y_2 (r, s)
={}&{} \mathbf y_2(r,s)y_1(c) - \sum_{q=0}^{s} \mathbf y_2(r-1, s-q) y_3(c+q); \label{eq:com.rel8} \\
x_3(b) \mathbf y_2 (r, s)
={}&{} \mathbf y_2(r,s)x_3(b) + \sum_{q=0}^{s} \mathbf y_2(r-1, s-q) x_1(b+q). \label{eq:com.rel9}
\end{align}
\end{lemma}
\begin{proof}
The proofs of equations~\eqref{eq:com.rel1}--\eqref{eq:com.rel7} are straight-forward.  The proof of equation~\eqref{eq:garland} is given in \cite[Lemma~1.3(ii)]{CP01}.  The proof of equation~\eqref{eq:com.rel8} is obtained by repeatedly using \eqref{eq:com.rel6} and the fact that $y_3y_2 - y_2y_3 = [y_3,y_2] = 0$, namely:
\begin{align*}
y_1(c)\mathbf y_2(r,s)
&= \sum_{\substack{b_0 + b_1 + \dotsb + b_s = r \\ b_1 + 2 b_2 + \dotsb + sb_s = s}} \left( y_2^{(b_0)} y_2(1)^{(b_1)} \dotsm y_2(s)^{(b_s)} y_1(c) \right. \\
& \qquad\qquad \left. - \sum_{q=0}^s y_2^{(b_0)} \dotsm y_2(q-1)^{(b_{_{q-1}})} y_2(q)^{(b_q-1)} y_3(c+q) y_2(q+1)^{(b_{_{q+1}})}\dotsm y_2(s)^{(b_s)} \right)\\
&= \mathbf y_2(r,s) y_1(c) - \sum_{q=0}^s \sum_{\substack{a_0 + a_1 + \dotsb + a_s = r-1 \\ a_1 + 2a_2 + \dotsb sa_s = s-q}} y_2^{(a_0)} \dotsm y_2(s)^{(a_s)} y_3(c+q) \\
&= \mathbf y_2(r,s) y_1(c) - \sum_{q=0}^s \mathbf y_2(r-1, s-q) y_3(c+q).
\end{align*}
The proof of equation~\eqref{eq:com.rel9} is analogous to the proof of equation~\eqref{eq:com.rel8}, as it is obtained by repeatedly using \eqref{eq:com.rel7} and the fact that $x_1y_2 - y_2x_1 = [x_1,y_2] = 0$.
\end{proof}

We now introduce a linear order on the following set of monomials in $\U(\g[t])$:
\[
\mathcal M := \{ x_1(b_1) \dotsm x_1(b_k) y_3(c_1) \dotsm y_3(c_\ell) \mid 0 \le b_1 < \dotsb < b_k, \ 0 \le c_1 < \dotsb < c_\ell, \ k, \ell \in \Z_{\ge0} \}.
\]
First, consider the short-reverse-colex order on tuples $(b_1, \dotsc, b_k)$, that is, $(b_1, \dotsc, b_k) \prec (b_1', \dotsc, b_q')$ if $k < q$, or $k=q$ and $(b'_k, \dotsc, b'_1) < (b_k, \dotsc, b_1)$ in the lexicographical order.  Then, induce a lexicographical-like order $\prec$ on $\mathcal M$ in the following way.  Given $\pmb b = (b_1, \dotsc, b_k)$, $\pmb c = (c_1, \dotsc, c_\ell)$, $\pmb b' = (b'_1, \dotsc, b'_p)$, $\pmb c' = (c'_1, \dotsc, c'_q)$, we say that $(\pmb b; \pmb c) \prec (\pmb b'; \pmb c')$ if $\pmb b \prec \pmb b'$ in the short-reverse-colex order, or $\pmb b = \pmb b'$ and $\pmb c \prec \pmb c'$ in the short-reverse-colex order.

Now, given $\pmb b = (b_1, \dotsc, b_k)$, $\pmb c = (c_1, \dotsc, c_\ell)$, $\pmb b' = (b'_1, \dotsc, b'_p)$, $\pmb c' = (c'_1, \dotsc, c'_q)$, finite (strictly) increasing sequences of positive integers, define $m(\pmb b; \pmb c) := x_1(b_1) \dotsm x_1(b_k) y_3(c_1) \dotsm y_3(c_\ell)$ and set
\begin{equation}\label{eq:total.order}
m(\pmb b; \pmb c) \prec m(\pmb b'; \pmb c')
\Leftrightarrow (\pmb b; \pmb c) \prec (\pmb b'; \pmb c').
\end{equation}
Notice that the minimal element in $\mathcal M$ (with respect to $\prec$) is the one with $k = \ell = 0$, which corresponds to $m(\emptyset; \emptyset) := 1$.  Define the \emph{weight} of an element $m(\pmb b; \pmb c) = m(b_1, \dotsc, b_k; c_1, \dotsc, c_\ell) \in \mathcal M$ to be $\wt (m(\pmb b; \pmb c)) := k + \ell$ and its \emph{degree} to be ${\rm dg}(m(\pmb b; \pmb c)) := b_1 + \dotsb + b_k + c_1 + \dotsb + c_\ell$.

The next result gives a generating set for graded local Weyl modules.  This set will be proven to be a basis in Theorem~\ref{thm:weyl.dim}.  We will need the following well-known binomial formula  
\begin{equation}\label{e:binomial}
\sum_{q=0}^p \binom{p}{q} a^{p-q} = (a+1)^p, \quad a \in \mathbb R, \ p \in \Z_{\ge0}.
\end{equation}

\begin{proposition} \label{prop:Weyl_basis}
For every $\lambda = (\lambda_1, \lambda_2) \in P^+$, the local Weyl module $W(\lambda)$ is spanned by
\begin{equation}\label{eq:span.set.sl(1,2)}
y_2(a_1) \dotsm y_2(a_j) x_1(b_1) \dotsm x_1(b_k) y_3(c_1) \dotsm y_3(c_\ell)w_\lambda,
\end{equation}
with $0 \leq c_1 < \cdots < c_\ell \leq \lambda_2-1$, $0 \leq b_1 < \cdots < b_k \leq \lambda_2-\ell-1$, $0 \leq a_1 \leq \cdots \leq a_j \leq \lambda_2-\ell-k-j$.  In particular, $\dim W (\lambda) \leq 4^{\lambda_2}$.
\end{proposition}
\begin{proof}
First, notice that the PBW Theorem together with the defining relations \eqref{eq:defn.rels.wpsi} imply that
\begin{align}
W(\lambda)
&= \U(\n^-[t]) \U(\h[t]) \U(\n^+[t])w_\lambda \notag \\
&= \U(\n^-[t])w_\lambda \notag \\
&= \sum_{\substack{a_1, \dotsc, a_j \ge 0 \\ b_1, \dotsc, b_k \ge 0 \\ c_1, \dotsc, c_\ell \ge 0}} \C \, y_2(a_1) \dotsm y_2(a_j) x_1(b_1) \dotsm x_1(b_k) y_3(c_1) \dotsm y_3(c_\ell)w_\lambda. \label{eq:1st.span}
\end{align}

Now, notice that $0 = [y_3, y_3] = 2 y_3^2$ and $0 = [x_1, x_1] = 2 x_1^2$.  Hence $y_3(a)^2 w_\lambda = x_1(a)^2 w_\lambda = 0$ for all $a \geq 0$; that is, $b_i \ne b_j$ and $c_i \ne c_j$ for all $i \ne j$ in \eqref{eq:1st.span}.  Thus, since $y_3(a) y_3(b) = - y_3(b)y_3(a)$ and $x_1(a) x_1(b) = - x_1(b) x_1(a)$ for all $a, b \ge 0$, we can reorder the factors in \eqref{eq:1st.span} in order to obtain $0 \le b_1 < \dotsb < b_k$ and $0 \le c_1 < \dotsb < c_\ell$.

Now, using equation~\eqref{eq:garland} and the relation $y_2^{\lambda_2+1} w_\lambda = 0$ \eqref{eq:defn.rels.wpsi}, we obtain that $y_2(a)w_\lambda = 0$ for all $a \ge \lambda_2$.  Hence $y_3(c) w_\lambda = [y_2(c), y_1] w_\lambda = 0$ for all $c \ge \lambda_2$.  Thus, we can restrict $c_\ell < \lambda_2$ in \eqref{eq:1st.span}.

Moreover, using \eqref{eq:com.rel1}, one sees that $x_2 \left( y_3(c_1) \dotsm y_3(c_\ell)w_\lambda \right) = 0$ and using \eqref{eq:com.rel2}, one sees that $h_2 \left( y_3(c_1) \dotsm y_3(c_\ell)w_\lambda \right) = (\lambda_2 - \ell) \left( y_3(c_1) \dotsm y_3(c_\ell)w_\lambda \right)$.  Hence, the representation theory of $\lie{sl}(2)$ implies that $y_2^{\lambda_2 - \ell + 1} \left( y_3(c_1) \dotsm y_3(c_\ell)w_\lambda \right) = 0$.  Using this vanishing condition and equation~\eqref{eq:garland}, we conclude that $y_2(a) \left( y_3(c_1) \dotsm y_3(c_\ell)w_\lambda \right) = 0$ for all $a \ge \lambda_2 - \ell$.  Hence $x_1(b) \left( y_3(c_1) \dotsm y_3(c_\ell)w_\lambda \right) = [x_3, y_2(b)] \left( y_3(c_1) \dotsm y_3(c_\ell)w_\lambda \right) = 0$ for all $b \ge \lambda_2 - \ell$.  Thus, we can restrict $b_k < \lambda_2 - \ell$ in \eqref{eq:1st.span}.

Since $c_\ell$ and $b_k$ are bounded, we have that the set of non-zero vectors of the form \eqref{eq:span.set.sl(1,2)} without $y_2(a_i)$ is finite. Let $\cF$ be the set of monomials $m \in \mathcal M$ for which $mw_\lambda \ne 0$.  Since $(\mathcal M, \prec)$ is a linearly ordered set and $\mathcal F$ is a finite subset of $\mathcal M$, then one can choose $N \in \Z_{\ge 0}$ and $m_0 \dotsc, m_N \in \mathcal M$ such that $\mathcal F = \{m_0, \dotsc, m_N\}$ and $m_0 \prec \dotsb \prec m_N$.  Now, for $r \in \Z$, define
	\begin{equation} \label{eq:the.filtration}
F(r) = \sum_{i < r} \U(\g_0[t])m_iw_\lambda.
	\end{equation}
Notice that: $F(r) = \{0\}$ for all $r \le 0$, and $F(r) = W(\lambda)$ for all $r > N$.  Also notice that $\{F(r)\}_{r \in \Z}$ is a filtration of $W(\lambda)$ as a $\g_0[t]$-module. Let ${\rm Gr}(W(\lambda))$ be the associated graded $\U(\g_0[t])$-module, and notice that ${\rm Gr}(W(\lambda))$ is isomorphic to $W(\lambda)$ as a vector space.  Hence, if we show that the images of the elements \eqref{eq:span.set.sl(1,2)} span ${\rm Gr}(W(\lambda))$, then \eqref{eq:span.set.sl(1,2)} will also span $W(\lambda)$.

Given a monomial $m_r = x_1(b_1) \dotsm x_1(b_k) y_3(c_1) \dotsm y_3(c_\ell)$, $r \in \{0, \dotsc, N\}$, we have:
\begin{align*}
x_2(a) \left(m_rw_\lambda\right) {}&{} \in F(r)  \quad \textup{for all $a \ge 0$}
& \textup{(by \eqref{eq:com.rel3}, \eqref{eq:com.rel1}, \eqref{eq:com.rel5})} \\
h(a) \left(m_rw_\lambda\right) {}&{} \in F(r) \quad \textup{for all $h \in \lie h$, $a > 0$}
& \textup{(by \eqref{eq:com.rel4}, \eqref{eq:com.rel2})} \\
h_2 \left(m_rw_\lambda\right) {}&{} = (\lambda(h_2) - \wt (m_r)) w_\lambda
& \textup{(by \eqref{eq:com.rel4}, \eqref{eq:com.rel2})}.
\end{align*}
Denote the $\U(\lie g_0[t])$-submodule generated by image of $m_rw_\lambda$ in ${\rm Gr}(W(\lambda))$ by $\overline{W}$, and  notice that the center of $\g_0[t]$ acts by a scalar on $m_r w_\lambda$ and recall that $\lie g'_0$ is isomorphic to $\lie {sl}(2)$, then the relations above imply that $\overline{W}$ is in fact isomorphic to a quotient of the local graded Weyl $\U(\lie{sl}(2)[t])$-module of highest-weight $\lambda_2 - \wt (m_r)$.  Using the basis for local graded Weyl $\U(\lie{sl}(2)[t])$-modules given in \cite[Section~6]{CP01}, we see that $\overline{W}$ is generated by elements of the form $y_2(a_1)\cdots y_2(a_j) \left(m_rw_\lambda\right)$ with $0 \leq a_1 \leq \dotsb \leq a_j \leq \lambda_2-k-\ell-j$.  Thus, the first statement follows.

In order to prove the second statement, note that the cardinality of the set of the monomials in \eqref{eq:span.set.sl(1,2)} is an upper bound to $\dim W(\lambda)$. Therefore 
$$\dim W(\lambda) \le \sum_{\ell = 0}^{\lambda_2} \sum_{k = 0}^{\lambda_2 - \ell} \sum_{j = 0}^{\lambda_2 - \ell - k} \binom{\lambda_2}{\ell} \binom{\lambda_2 - \ell}{k} \binom{\lambda_2 - \ell -  k}{j} = 4^{\lambda_2}$$
where the last equality follows by repeatedly using \eqref{e:binomial}.
\end{proof}

Notice that the result of Proposition~\ref{prop:Weyl_basis} also holds if we exchange $y_3$ and $x_1$ in the monomials \eqref{eq:span.set.sl(1,2)}.  That is, $W(\lambda)$ is also spanned by the following elements
\begin{equation*}
y_2(a_1) \dotsm y_2(a_j) y_3(b_1) \dotsm y_3(b_k) x_1(c_1) \dotsm x_1(c_\ell)w_\lambda,\end{equation*}
where $0 \leq c_1 < \cdots < c_\ell \leq \lambda_2-1$, $0 \leq b_1 < \cdots < b_k \leq \lambda_2-\ell-1$, $0 \leq a_1 \leq \cdots \leq a_j \leq \lambda_2-\ell-k-j$.

\subsection{Fusion products} \label{ss:fusion}

Consider the grading $\U(\g[t]) = \bigoplus_{s \ge 0} \U(\g[t])[s]$ induced by the usual grading of $\C[t]$ given by degrees in $t$, that is, $\U(\g[t])[s]$ is defined to be the subspace of $\U(\g[t])$ spanned by elements of the form $X_1(a_1) \dotsm X_r(a_r)$, where $X_1, \dotsc, X_r \in \g$ and $a_1 + \dotsb + a_r = s$.  This grading induces a filtration on any cyclic $\g[t]$-module.  In fact, let $V$ be a cyclic $\g[t]$-module and $v \in V$ be a cyclic vector.  For each $n \ge 0$, define
\begin{equation}\label{eq:filt.cycl.mod}
V^{\leq n} := \sum_{i = 0}^n \U(\g[t])[i] v.
\end{equation}
Notice that $V^{\leq 0} = \U(\g)v$, $V^{\leq n} \subseteq V^{\leq n+1}$ (for all $n \ge 0$), and $V =\sum_{n\ge 0} V^{\leq n}$.  (Notice that this filtration depends on the choice of $v$.)  Now, for each $n > 0$, let $V[n] := V^{\leq n}/V^{\leq n-1}$, let $V[0] = V^{\le 0}$, and let ${\rm Gr}_v(V) := \bigoplus_{n \ge 0} V[n]$ be the associated graded $\g[t]$-module.

Given a $\g$-module $V$ and a complex number $z\in \C$, let $V^z$ denote the associated evaluation module, that is, the $\g[t]$-module with underlying vector space $V$ and with $x(a) \in \g[t]$ acting on $v \in V$ by $x(a) v = z^a (x v)$.

If $V_1,\dotsc, V_n$ are cyclic $\g$-modules with cyclic homogeneous even vectors $v_1, \dotsc, v_n$, respectively, and $z_1, \dotsc, z_n$ are pairwise-distinct complex numbers, then the tensor product $V = V_1^{z_1} \otimes \dotsb \otimes V_n^{z_n}$ is a cyclic $\g[t]$-module with cyclic homogeneous even vector $v_1 \otimes \dotsb \otimes v_n$.

Using this notation, we now recall the definition of fusion products.  These modules are key to our study of local Weyl modules (see Theorem~\ref{thm:weyl.dim}), Chari-Venkatesh modules (see Theorem~\ref{thm:iso.vcsi}), as well as related to a conjecture of Feigin and Loktev \cite[Conjecture~1.8(i)]{FL99} (see Corollary~\ref{cor:iso.vcsi}\ref{cor:CV.indep}).

\begin{definition}
Let $V_1,\dotsc, V_n$ be cyclic $\g$-modules with cyclic vectors $v_1, \dotsc, v_n$ (respectively) and $z_1, \dotsc, z_n$ be pairwise-distinct complex numbers.  The \emph{fusion product} of $V_1, \dotsc, V_n$ with \emph{fusion parameters} $z_1, \dotsc, z_n$ is defined to be the graded $\g[t]$-module ${\rm Gr}_{v_1 \otimes \dotsb \otimes v_n} (V_1^{z_1} \otimes \dotsb \otimes V_n^{z_n})$.  We denote ${\rm Gr}_{v_1 \otimes \dotsb \otimes v_n} (V_1^{z_1} \otimes \dotsb \otimes V_n^{z_n})$ by $V^{z_1}_1 \ast \dotsm \ast V^{z_n}_n$ and any vector $w_1 \otimes \dotsb \otimes w_n \in V_1^{z_1} \otimes \dotsb \otimes V_n^{z_n}$ will be denoted by $w_1 * \dotsb * w_n$ when regarded as a vector of $V^{z_1}_1 \ast \dots \ast V^{z_n}_n$.
\end{definition}

Observe that, for every $n > 0$ and every permutation $\sigma$ in the symmetric group $S_n$, the isomorphism between the $\g[t]$-modules $V^{z_1}_1 \otimes \dots \otimes V^{z_n}_n$ and $V^{z_{\sigma(1)}}_{\sigma(1)} \otimes \dots \otimes V^{z_{\sigma(n)}}_{\sigma(n)}$ induces an isomorphism between the $\g[t]$-modules $V^{z_1}_1 \ast \dots \ast V^{z_n}_n$ and $V^{z_{\sigma(1)}}_{\sigma(1)} \ast \dots \ast V^{z_{\sigma(n)}}_{\sigma(n)}$.

\begin{remark} \label{rmk:vanishing.fusion}
In the particular case where $V = V_1^{z_1} \otimes \dotsb \otimes V_n^{z_n}$, $V_1,\dotsc, V_n$ are cyclic $\g$-modules with cyclic homogeneous even vectors $v_1, \dotsc, v_n$ (respectively), $v = v_1 \otimes \dotsb \otimes v_n$, and $z_1, \dotsc, z_n$ are pairwise-distinct complex numbers, one can check that, for all homogeneous element $x \in \g$ and $s \ge n$, we have
\[
(x \otimes t^s) v
= (x \otimes t^{s-n}(t-z_1)\dotsm(t-z_n))v
= 0 \ \in \ V^{z_1}_1 \ast \dotsm \ast V^{z_n}_n.
\qedhere
\]
\end{remark}

\subsection{The structure of local Weyl modules}

In this subsection we describe the structure of local Weyl modules.  We begin with a result relating graded local Weyl modules to fusion products.

\begin{theorem} \label{thm:weyl.dim}
Let $\lambda = (\lambda_1, \lambda_2) \in P^+$, $z_1, \dotsc, z_{\lambda_2}$ be pairwise-distinct complex numbers, and $\kappa_1, \dotsc, \kappa_{\lambda_2}$ be complex numbers satisfying $\kappa_1 + \dotsb + \kappa_{\lambda_2} = \lambda_1$.  There is an isomorphism of $\g[t]$-modules
\[
W(\lambda) \cong K(\kappa_1,1)^{z_1} \ast \dotsb \ast K(\kappa_{\lambda_2},1)^{z_{\lambda_2}}.
\]
In particular, $\dim W(\lambda) = 4^{\lambda_2}$.
\end{theorem}

\begin{proof}
Recall from Proposition~\ref{prop:Weyl_basis} that $\dim W(\lambda) \le 4^{\lambda_2}$ and from Proposition~\ref{prop:dim.k2.sl12} that $\dim K(\mu) = 4\mu_2$ (since we have fixed $\lie b = \lie b_{\scriptscriptstyle (2)}$).  Hence, if we prove that there is a surjective homomorphism of $\g[t]$-modules $W(\lambda) \twoheadrightarrow K(\kappa_1,1)^{z_1} \ast \dotsb \ast K(\kappa_{\lambda_2},1)^{z_{\lambda_2}}$, the result follows.

For each $i \in \{1, \dotsc, \lambda_2\}$, let $\mu_i$ denote the weight $(\kappa_i, 1)\in P^+$.  Notice that, in order to prove that there is a surjective homomorphism of $\g[t]$-modules $W(\lambda) \twoheadrightarrow K(\mu_1)^{z_1} \ast \dotsb \ast K(\mu_{\lambda_2})^{z_{\lambda_2}}$, it suffices to show that the vector $k := k_{\mu_1} * \dotsb * k_{\mu_{\lambda_2}}$ satisfies the defining relations \eqref{eq:defn.rels.wpsi}.

Using the comultiplication of $\U(\g[t])$, one can check that
	\[
\n^+[t]k = 0, \qquad
hk = \lambda(h)k \quad \textup{for all $h \in \h$}, \qquad
y_2^{\lambda_2+1} k = 0.
	\]
Moreover, for all $h \in \h$ and $s > 0$, we have:
\[
(h \otimes t^s) (k_{\mu_1} \otimes \dotsb \otimes k_{\mu_{\lambda_2}})
= \sum_{i=1}^{\lambda_2} k_{\mu_1} \otimes \dotsb \otimes (h \otimes t^s) k_{\mu_i} \otimes \dotsb \otimes k_{\mu_{\lambda_2}}
= \sum_{i=1}^{\lambda_2} z_i^s \mu_i(h) \ (k_{\mu_1} \otimes \dotsb \otimes k_{\mu_{\lambda_2}}).
\]
Since $(h \otimes t^s) \in \U(\g[t])[s]$, $k_{\mu_1} \otimes \dotsb \otimes k_{\mu_{\lambda_2}} \in V[0]$ and $(h \otimes t^s) (k_{\mu_1} \otimes \dotsb \otimes k_{\mu_{\lambda_2}}) \in V^{\le s-1}$, it follows that, in the fusion product, we have $(h \otimes t^s) k = 0$.
\end{proof}

Now we can use Theorem~\ref{thm:weyl.dim} to obtain interesting results on a more general class of local Weyl modules.  Notice that these local Weyl modules are not necessarily graded.  In fact, they are precisely those local Weyl modules whose associated graded modules are isomorphic to graded local Weyl modules.

\begin{lemma} \label{lem:dim.wpsi}
Let $\psi \in \h[t]^*$, assume that there exist $\mu_1, \dotsc, \mu_n \in P^+$ and pairwise-distinct complex number $z_1, \dotsc, z_n$ such that $\psi(h \otimes t^k) = z_1^k \mu_1(h) + \dotsb + z_n^k \mu_n(h)$ for all $h \in \h$, $k \ge 0$.  For each $i \in \{1, \dotsc, n\}$, denote by $\psi_i$ the unique linear functional in $\h[t]^*$ satisfying $h \otimes t^k \mapsto z_i^k \mu_i(h)$ for all $h \in \h$, $k \ge 0$.
\begin{enumerate}[(a), leftmargin=*]
\item \label{lem:dim.wpsi.a}
There exists an isomorphism of $\g[t]$-modules $W(\psi)\cong W(\psi_1) \otimes \dotsm \otimes W(\psi_n)$.

\item \label{lem:dim.wpsi.b}
For each $i \in \{1, \dotsc, n\}$, there exists an isomorphism of $\g$-modules $W(\psi_i) \cong W(\mu_i)$.

\item \label{lem:dim.wpsi.c}
$\dim W(\psi) = 4^{\psi(h_2)}$.
\end{enumerate}
\end{lemma}
\begin{proof}
Recall first that for a ${\lie g}[t]$-module $M$, we define ${\rm Ann}_{\C[t]} M$ to be the sum of all ideals $I$ of $\C[t]$ such that $({\lie g}\otimes I)M=0$. Moreover, the support of $M$ is defined to be the set of all maximal ideals ${\mathsf m}$ of $\C[t]$ such that ${\rm Ann}_{\C[t]} M\subseteq {\mathsf m}$. Part~\ref{lem:dim.wpsi.a} follows from \cite[Proposition~8.18]{BCM19} along with the fact that the support of $W(\psi_i)$ equals the maximal ideal $(t-z_i)$ for all $i=1,\ldots, n$.

In order to prove part~\ref{lem:dim.wpsi.b}, for each $z \in \C$, let $\tau_z \colon \g[t] \to \g[t]$ be the unique linear map that satisfies $\tau_z (x \otimes t^k) = x \otimes (t-z)^k$ for all $x \in \g$, $k \in \Z_{\ge0}$.  Notice that $\tau_z$ is in fact an automorphism of Lie algebras for all $z \in \C$ ($\tau_z^{-1}$ being $\tau_{-z}$).  Given a representation $\rho \colon \g[t] \to \lie{gl}(V)$ denote by $\tau_z^*V$ the $\g[t]$-module corresponding to the representation $(\rho \circ \tau_z) \colon \g[t] \to \lie{gl}(V)$.  Since $\tau_z \arrowvert_\g = {\rm id}_\g$, it follows that $V \cong \tau_z^*V$ as $\g$-modules.  Now, to finish the proof of part~\ref{lem:dim.wpsi.b}, one can check that $W(\mu_i) \cong \tau_{z_i}^* W(\psi_i)$ for all $i \in \{1, \dotsc, n\}$.

Part~\ref{lem:dim.wpsi.c} follows from part~\ref{lem:dim.wpsi.a}, part~\ref{lem:dim.wpsi.b}, Theorem~\ref{thm:weyl.dim} and the fact that 
\[
\mu_1(h_2) + \dotsb + \mu_n(h_2)
= \psi_1(h_2) + \dotsb + \psi_n(h_2)
= \psi(h_2).
\qedhere
\]
\end{proof}

Recall, from Section~\ref{ss:fusion}, the definition of associated graded modules, such as ${\rm Gr}_{w_\psi} W(\psi)$.  The next result relates different local Weyl modules.  Coupled with Theorem~\ref{thm:weyl.dim}, it will enable us to describe the structure of a large class of local Weyl modules.

\begin{theorem} \label{thm:gr.wpsi.weyl}
Let $\psi \in \h[t]^*$, assume that $\psi \arrowvert_\h \in P^+$ and denote it by $\lambda$.  There exists a surjective homomorphism of $\g[t]$-modules:
\[
W(\lambda) \twoheadrightarrow {\rm Gr}_{w_\psi} W(\psi)
\quad \textup{given by} \quad
w_\lambda \mapsto w_\psi.
\]
Moreover, such a homomorphism is actually an isomorphism  if there exist $\mu_1, \dotsc, \mu_n \in P^+$ and pairwise-distinct $z_1, \dotsc, z_n \in \C$ such that $\psi(h \otimes t^k) = z_1^k \mu_1(h) + \dotsb + z_n^k \mu_n(h)$ for all $h \in \h$ and $k \ge 0$.
\end{theorem}
\begin{proof}
In order to prove the first statement, it is enough to check that $w_\psi \in {\rm Gr}_{w_\psi} W(\psi)$ satisfies the defining relations \eqref{eq:defn.rels.wpsi} of $w_\lambda$.  It follows from the definition of $W(\psi)$ that
\[
\n^+[t] w_\psi = 0, \qquad 
h w_\psi = \psi(h) w_\psi \quad \textup{for all $h \in \h$},
\qquad \textup{and} \qquad
y_2^{\psi(h_2)+1} w_\psi = 0.
\]
Moreover, for all $h \in \h$ and $s > 0$, we have
\[
(h \otimes t^s) w_\psi
= \psi(h \otimes t^s) w_\psi \in \sum_{i = 0}^{s-1} \U(\g[t])[i] w_\psi.
\]
Hence, in ${\rm Gr}_{w_\psi} W(\psi)$, we have that $(h \otimes t^s) w_\psi = 0$.  This proves the first part.

Now, since $\dim {\rm Gr}_{w_\psi} W(\psi) = \dim W(\psi)$ and $\dim W(\lambda) = 4^{\lambda_2}$ (Theorem~\ref{thm:weyl.dim}), the first part implies that $\dim W(\psi) \le 4^{\lambda_2}$.  Under the hypothesis of the \emph{moreover} part, we have from Lemma~\ref{lem:dim.wpsi}\ref{lem:dim.wpsi.c} that $\dim W(\psi) = 4^{\psi(h_2)}$.  Since $\lambda_2 = \lambda(h_2) = \psi(h_2)$, the result follows.
\end{proof}

\begin{corollary} \label{cor:weyl.dim}
Let $\psi \in \h[t]^*$, and assume that there exist $\mu_1, \dotsc, \mu_n \in P^+$ and pairwise-distinct $z_1, \dotsc, z_n \in \C$ such that $\psi(h \otimes t^k) = z_1^k \mu_1(h) + \dotsb + z_n^k \mu_n(h)$ for all $h \in \h$, $k \ge 0$.

\begin{enumerate}[(a), leftmargin=*, itemsep=1ex]
\item \label{cor:basis.weyl}
The elements $y_2(a_1) \dotsm y_2(a_j) x_1(b_1) \dotsm x_1(b_k) y_3(c_1) \dotsm y_3(c_\ell) w_\psi$, such that $0 \leq c_1 < \cdots < c_\ell \leq \psi(h_2)-1$, $0 \leq b_1 < \cdots < b_k \leq \psi(h_2)-\ell-1$, $0 \leq a_1 \leq \cdots \leq a_j \leq \psi(h_2)-\ell-k-j$, form a basis of $W(\psi)$.

\item \label{cor:weyl.dim.b}
The character of $W(\psi)$ is given by
	\[
\ch W(\psi) = e^{(\psi(h_1),\, 0)} \left( e^{(-1,-1)} + e^{(-1,0)} + e^{(0,0)} + e^{(0,1)} \right)^{\psi(h_2)}.
	\]

\item \label{cor:embeddings.weyl}
If we denote $\psi \arrowvert_\h = \lambda = (\lambda_1, \lambda_2) \in P^+$, then there are well-defined embeddings of $\g[t]$-modules:
\begin{gather}
W(\lambda) \hookrightarrow W(\lambda_1+1, \, \lambda_2+1)
\quad \textup{given by} \quad
w_\lambda \mapsto y_3(\lambda_2)w_{_{(\lambda_1+1, \lambda_2+1)}}  \label{eq:embedding.y3}, \\
W(\lambda) \hookrightarrow W(\lambda_1, \, \lambda_2+1)
\quad \textup{given by} \quad
w_\lambda \mapsto x_1(\lambda_2)w_{_{(\lambda_1, \lambda_2+1)}}.  \label{eq:embedding.x1}
\end{gather}
\end{enumerate}
\end{corollary}
\begin{proof}
Part~\ref{cor:basis.weyl} follows from Theorem~\ref{thm:weyl.dim} and Theorem~\ref{thm:gr.wpsi.weyl}.

To prove part~\ref{cor:weyl.dim.b}, notice that it follows from Lemma~\ref{lem:kac.iso} that $\ch K(\kappa_{\lambda_i}, 1) = \ch K_{{\lie b}_{\scriptscriptstyle (1)}}(\kappa_{\lambda_i}, 0)$ or $\ch K(\kappa_{\lambda_i}, 1) = \ch K_{{\lie b}_{\scriptscriptstyle (3)}}(\kappa_{\lambda_i}-1, 0)$.  Now, for all $\mu \in \h^*$, we have:
	\[
\ch K_{{\lie b}_{\scriptscriptstyle (1)}}(\mu) = \ch L_{\gl}(\mu) \ch \Lambda(\g_{-1}) =  \ch L_{\gl}(\mu) \left( e^{(0,0)}+e^{(0,1)}+ e^{(-1,-1)}+e^{(-1,0)} \right)
	\]
and 
	\[
\ch K_{{\lie b}_{\scriptscriptstyle (3)}}(\mu) = \ch L_{\gl}(\mu) \ch \Lambda(\g_{1}) =  \ch L_{\gl}(\mu) \left( e^{(0,0)}+e^{(0,-1)} + e^{(1,1)}+e^{(1,0)} \right). 
	\]
Then part~\ref{cor:weyl.dim.b} follows from Theorem~\ref{thm:weyl.dim}, Theorem~\ref{thm:gr.wpsi.weyl}, and a straight-forward calculation.

In order to prove \eqref{eq:embedding.y3}, first, one can check that the element $y_3(\lambda_2)w_{_{(\lambda_1+1, \lambda_2+1)}} \in W(\lambda_1+1, \lambda_2+1)$ satisfies the defining relations \eqref{eq:defn.rels.wpsi} of $w_\lambda$.  Thus, there is a well-defined homomorphism of $\g[t]$-modules $W(\lambda) \to W(\lambda_1+1, \, \lambda_2+1)$ mapping $w_\lambda$ to $y_3(\lambda_2)w_{_{(\lambda_1+1, \lambda_2+1)}}$.  Moreover, using the explicit basis \eqref{eq:span.set.sl(1,2)} for $W(\lambda_1+1, \lambda_2+1)$, one can see that the $\g[t]$-submodule $\U(\g[t])y_3(\lambda_2)w_{_{(\lambda_1+1, \lambda_2+1)}}$ contains non-zero images of all the basis elements of $W(\lambda)$.  This proves \eqref{eq:embedding.y3}.  The proof of \eqref{eq:embedding.x1} is similar.
\end{proof}

\begin{remark}
Notice that the hypotheses of Lemma~\ref{lem:dim.wpsi}, Theorem~\ref{thm:gr.wpsi.weyl} and Corollary~\ref{cor:weyl.dim} cover all finite-dimensional non-zero local Weyl modules but those whose irreducible quotient is not an evaluation module.  The key for proving the remaining cases could be a better description of the finite-dimensional irreducible generalized evaluation modules that are not evaluation modules.
\end{remark}

\section{Chari-Venkatesh modules} \label{s:v.csi}

In this section, our goal is to better understand the structure of local Weyl modules and the relationship between these modules and general fusion products of Kac modules.  In order to do that, we define analogues of certain modules introduced by Chari and Venkatesh in \cite[Section~2.2]{CV15}.

\begin{definition} \label{defn:vcsi}
Let $\lambda \in P^+$.  Given a partition $\xi = (\xi_0, \dotsc, \xi_d)$ of $\lambda_2 \in \Z_{\ge0}$ (that is, a sequence of positive integers such that $\xi_0 + \dotsb + \xi_d = \lambda_2$ and $\xi_0 \ge \dotsb \ge \xi_d$), denote the pair $(\lambda_1, \xi)$ by $\pmb \xi$, and define the \emph{Chari-Venkatesh module} $V(\pmb\xi)$ to be the quotient of the local Weyl module $W (\lambda)$ by the extra relations
\begin{equation} \label{eq:def.rel.vcsi}
(x_2 \otimes t)^s y_2^{r+s} \qquad
\textup{for all} \quad s > kr + \xi_{k+1} + \dotsb + \xi_d, \ 
r, s > 0, \ d \ge k \ge 0.
\end{equation}
Notice that $V(\pmb\xi)$ is a cyclic $\g[t]$-module generated by the image of $w_\lambda$ onto $V(\pmb\xi)$, which we denote by $v_{\pmb\xi}$.
\end{definition}

\begin{lemma} \label{lem:vcsi.basc}
Let $\lambda = (\lambda_1, \lambda_2) \in P^+$.
\begin{enumerate}[(a), leftmargin=*]
\item \label{lem:vcsi.basc.a} 
If $\xi = (\lambda_2)$, then $V(\pmb\xi) = K (\lambda)^0$.

\item \label{lem:vcsi.basc.b}
If $\xi = (1, \dotsc, 1) =: (\underline 1^{\lambda_2})$, then $V(\pmb\xi) = W(\lambda)$.
\end{enumerate}
\end{lemma}
\begin{proof}
Recall that $V(\pmb\xi)$ is the quotient of the local Weyl module $W(\lambda)$ by the extra relations \eqref{eq:def.rel.vcsi}.

\begin{enumerate}[(a), leftmargin=*]
\item Assume $\xi = (\lambda_2)$.  Notice that, in this case, $d = 0$.  Thus, we can only choose $k = 0$.  Hence the extra relations are
\[
(x_2 \otimes t)^s y_2^{r+s} = 0 \qquad
\textup{for all} \quad r, s > 0.
\]
In particular, $(x_2 \otimes t)^s y_2^{s+1} = 0$ for all $s \ge 1$.  Using equation~\eqref{eq:garland}, this implies that $(y_2 \otimes t^s) = 0$ for all $s \ge 1$.  Hence, $V(\pmb\xi)$ is the quotient of $M^0$ for some $\g$-module $M$ that is a quotient of $K(\lambda)$.  Now, using equation~\eqref{eq:garland} again, one can check that there are no other extra relations.  Hence $M = K(\lambda)$.

\item In this case, we have $s > \lambda_2-1, \ s > r + \lambda_2 - 2, \ \dotsc, \ s > (\lambda_2-1)r + 1$.  Notice that, since $r \ge 1$, all these inequalities imply $s \ge \lambda_2$ and $r \ge 1$.  However, the relations $(x_2 \otimes t)^a y_2^b = 0$ for $b \ge \lambda_2 + 1$ follow from the usual representation theory of $\lie{sl}(2)$-modules.  That is, these are not actually extra relations.  Hence $V(\pmb\xi) = W (\lambda)$.
\qedhere
\end{enumerate}
\end{proof}

The following lemma gives a refinement of the extra relations \eqref{eq:def.rel.vcsi} (compare it with \cite[Theorems~2.1.1, 3.1.1, 3.2.1]{Murray18} for the case of $\mathfrak{sl}(2)[t]$).

\begin{lemma} \label{lem:vcsi.equi.defn}
Let $\lambda \in P^+$, $\xi = (\xi_0, \dotsc, \xi_d)$ be a partition of $\lambda_2 \in \Z_{\ge0}$, and $\pmb\xi := (\lambda_1, \xi)$.

\begin{enumerate}[(a), leftmargin=*]\itemsep1ex
\item \label{lem:vcsi.equi.defn.a}
The extra relations \eqref{eq:def.rel.vcsi} of $V(\pmb\xi)$ reduce to:
\begin{equation*}
(x_2 \otimes t)^s y_2^{r+s} v_{\pmb\xi} = 0
\quad \textup{for all} \quad
s > kr + \xi_{k+1} + \dotsb + \xi_d, \quad \xi_{k+1} \le r < \xi_k, \quad 0 < r, \quad 0 \le k \le d.
\end{equation*}

\item \label{lem:vcsi.equi.defn.b}
For all $r, s > 0$, the extra relations \eqref{eq:def.rel.vcsi} of $V(\pmb\xi)$ are equivalent to:
\[
\mathbf y_2 (r,s) v_{\pmb\xi} = 0
\quad \textup{for all} \quad
s > kr + \xi_{k+1} + \dotsb + \xi_d, \quad
\xi_{k+1} \le r < \xi_k, \quad
0 \le k \le d.
\]

\item \label{lem:vcsi.equi.defn.c} 
Moreover, the extra relations \eqref{eq:def.rel.vcsi} of $V(\pmb\xi)$ are equivalent to:
\[
\sum_{\substack{b_k + b_{k+1} + \dotsb + b_s = r \\ kb_k + (k+1) b_{k+1} + \dotsb + s b_s = s}} (y_2 \otimes t^k)^{(b_k)} (y_2 \otimes t^{k+1})^{(b_{k+1})} \dotsm (y_2 \otimes t^s)^{(b_s)} v_{\pmb\xi} = 0,
\]
for all $s > kr + \xi_{k+1} + \dotsb + \xi_d$, $\xi_{k+1} \le r < \xi_k$, $0 < r$, and $0 \le k \le d$.
\end{enumerate}
\end{lemma}
\begin{proof}
The proof of part~\ref{lem:vcsi.equi.defn.a} follows from the fact that $r + \xi_1 + \dotsb + \xi_d \ge \xi_0 + \dotsb + \xi_d = \lambda_2$ for all $r \ge \xi_0$, and the fact that $r + \xi_{k+1} + \dotsb + \xi_d \ge \xi_k + \xi_{k+1} + \dotsb + \xi_d$ for all $r \ge \xi_k$, $k \in \{1, \dotsc, d\}$.  The proofs of parts~\ref{lem:vcsi.equi.defn.b}~and~\ref{lem:vcsi.equi.defn.c} follow from \ref{lem:vcsi.equi.defn.a} and \cite[Lemma~2.3 and Proposition~2.6]{CV15}.
\end{proof}

The next result is our first step towards understanding the structure and constructing basis for Chari-Venkatesh modules (compare it with Theorem~\ref{thm:weyl.dim} and \cite[Proposition~4.6]{Kus18}).

\begin{proposition} \label{prop:vcsi.onto.kacs}
Let $\lambda = (\lambda_1, \lambda_2) \in P^+$, $\xi = (\xi_0, \dotsc, \xi_d)$ be a partition of $\lambda_2$, and $\pmb\xi := (\lambda_1, \xi)$.  For every choice of pairwise-distinct complex numbers $z_0, \dotsc, z_d$ and complex numbers $\kappa_0, \dotsc \kappa_d$ satisfying $\kappa_0 + \dotsb + \kappa_d = \lambda_1$, there is an epimorphism of $\g[t]$-modules
\[
V (\pmb\xi) \twoheadrightarrow K(\kappa_0, \xi_0)^{z_0} * \dotsb * K(\kappa_d, \xi_d)^{z_d}.
\]
\end{proposition}
\begin{proof}
This proof follows from arguments similar to the ones used in the proof of Theorem~\ref{thm:weyl.dim} and \cite[Proposition~6.8]{CV15}.
\end{proof}

\subsection{A few technical lemmas}

In order to show that the epimorphism of Proposition~\ref{prop:vcsi.onto.kacs} is in fact an isomorphism (Theorem~\ref{thm:iso.vcsi}) we   need a few technical lemmas.  

Given a partition $\xi = (\xi_0, \dotsc, \xi_d)$ of $n \in \Z_{\ge0}$, let $|\xi| := d+1$.  For each $c \in \{0, \dotsc, d\}$, define the following partition of $n - 1$:
\[
\varphi_c(\xi)
:= \begin{cases}
(\xi_0, \dotsc, \xi_{c-1}, \xi_c - 1, \xi_{c+1}, \dotsc, \xi_d), & \textup{if $\xi_c > \xi_{c+1}$}, \\
(\xi_0, \dotsc, \xi_{c-1}, \, \xi_{c+1}, \dotsc, \xi_{c+i}, \, \xi_c - 1, \, \xi_{c+i+1}, \dotsc, \xi_d), & \textup{if $\xi_c = \dotsb = \xi_{c+i} > \xi_{c+i+1}$, $i \ge 1$}.
\end{cases}
\]
Now, define $\mathcal J (\xi) := \{(c_1, \dotsc, c_\ell) \mid 0 \le \ell \le |\xi|, \ 0 \le c_1 < \dotsb < c_\ell < |\xi|\}$.  If $\ell = 0$ we denote $(c_1, \dotsc, c_\ell)$ by $\emptyset$ in $\mathcal J(\xi)$.  For each $\pmb c = (c_1, \dotsc, c_\ell) \in \mathcal J(\xi)$ with $\ell > 0$, define a partition of $n - \ell$ by $\pmb c(\xi) := \varphi_{c_1} \circ \dotsb \circ \varphi_{c_\ell}(\xi)$.  For $\pmb c = \emptyset$, define $\emptyset(\xi) := \varphi_\emptyset (\xi) := \xi$. Define also
\[
\mathcal I(\xi) 
= \{(b_1, \dotsc, b_k; \pmb c) \mid \pmb c \in \mathcal J (\xi), \ (b_1, \dotsc, b_k) \in \mathcal J(\pmb c(\xi)) \}.
\]
For each pair $(\pmb b; \pmb c) \in \mathcal I (\xi)$, define $(\pmb b;\pmb c)(\xi) := \pmb b(\pmb c(\xi)) := \varphi_{b_1}\circ \dotsb \circ \varphi_{b_k}\circ \varphi_{c_1}\circ \dotsb \circ \varphi_{c_\ell} (\xi)$.

Given a partition $\xi = (\xi_0, \dotsc, \xi_d)$ of $n \in \Z_{\ge0}$, for each $t \in \{1, \dotsc, d\}$, define
\[
\xi^\sharp_t := (\xi_0, \dotsc, \xi_{t-1})
\qquad \textup{and} \qquad 
\xi^\flat_t := (\xi_t, \dotsc, \xi_d).
\]
Notice that $\xi^\sharp_t$ is a partition of $\xi_0 + \dotsb + \xi_{t-1} < n$ and $\xi^\flat_t$ is a partition of $\xi_t + \dotsb + \xi_d < n$.  Moreover, for each $\pmb c = (c_1, \dotsc, c_\ell) \in \mathcal J (\xi^\sharp_t)$ and $\pmb z = (z_1, \dotsc, z_q) \in \mathcal J(\xi^\flat_t)$, notice that $\pmb c^{\pmb z} := (c_1, \dotsc, c_\ell, z_1+t, \dotsc, z_q+t) \in \mathcal J(\xi)$. In fact, $\ell + q \le |\xi^\sharp_t| + |\xi^\flat_t| = t+(d-t+1) = d+1 = |\xi|$, and $0 \le c_1 < \dotsb < c_\ell < |\xi^\sharp_t| = t \le z_1 + t < \dotsb < z_q + t < |\xi^\flat_t| + t = |\xi|$.  On the other hand, a similar argument shows that for each $\pmb e = (e_1, \dotsc, e_m) \in \mathcal J(\xi)$ there exists $\ell \in \{0, \dotsc, m\}$ such that $e_i < t$ for all $i \le \ell$ and $e_i \ge t$ for all $i > \ell$.  In the case $\ell = 0$, choose $\pmb c = \emptyset$, $\pmb z = \pmb e$, and notice that $\pmb c^{\pmb z} = \pmb e$.  In the case $\ell = m$, choose $\pmb c = \pmb e$, $\pmb z = \emptyset$, and notice that $\pmb c^{\pmb z} = \pmb e$.  Otherwise, choose $\pmb c := (e_1, \dotsc, e_\ell) \in \mathcal J(\xi^\sharp_t)$, $\pmb z := (e_\ell-t, \dotsc, e_m-t) \in \mathcal J(\xi^\flat_t)$, and notice that $\pmb c^{\pmb z} = \pmb e$.  This argument shows that
\begin{equation} \label{eq:dec.J}
\mathcal J(\xi) 
= \{ \pmb c^{\pmb z} \mid \pmb c \in \mathcal J(\xi^\sharp_t), \ \pmb z \in \mathcal J(\xi^\flat_t)\}.
\end{equation}
In this first lemma, we prove the analogous result for $\mathcal I(\xi)$ under a mildly restrictive condition.

\begin{lemma} \label{lem:dec.I}
Let $\xi = (\xi_0, \dotsc, \xi_d)$ be a partition of $n \in \Z_{\ge0}$.  If there exists $t \in \{1, \dotsc, d\}$ such that $\xi_{t-1} > \xi_t$, then $\mathcal I(\xi) = \{ \left( \pmb b^{\pmb u}; \pmb c^{\pmb z} \right) \mid \pmb{c} \in \mathcal J(\xi^\sharp_t), \ \pmb{z} \in \mathcal J(\xi^\flat_t), \  \pmb b \in \mathcal J(\pmb{c}(\xi^\sharp_t)), \ \pmb u \in \mathcal J(\pmb z(\xi^\flat_t)) \}$.
\end{lemma}
\begin{proof}
By \eqref{eq:dec.J}, we can write $\mathcal I (\xi) = \{ (\pmb e; \pmb c^{\pmb z}) \mid \pmb c \in \mathcal J(\xi^\sharp_t), \ \pmb z \in \mathcal J(\xi^\flat_t), \ \pmb e \in \mathcal J(\pmb c^{\pmb z}(\xi)) \}$.  Since, by hypothesis, $1 \le \xi_t < \xi_{t-1}$, we have that $\pmb c^{\pmb z}(\xi)_{t-1} \le \pmb c^{\pmb z}(\xi)_t$.  Hence $\pmb c^{\pmb z}(\xi)^\sharp_t = \pmb c(\xi^\sharp_t)$, $\pmb c^{\pmb z}(\xi)^\flat_t = \pmb z(\xi^\flat_t)$ and $\mathcal J(\pmb c^{\pmb z}(\xi)) = \{ \pmb b^{\pmb u} \mid \pmb b \in \mathcal J(\pmb c(\xi^\sharp_t)), \ \pmb u \in \mathcal J(\pmb z(\xi^\flat_t)) \}$.  The result follows.
\end{proof}

Recall from Section~\ref{s:kac} that we denote by $L_{\lie{sl}}(\nu)$ the finite-dimensional irreducible $\g'_0$-module of highest weight $\nu \in (\h')^*$.  This next lemma will be used in the proof of Theorem~\ref{thm:iso.vcsi}.

\begin{lemma} \label{lem:dim.fil}
For every partition $\xi = (\xi_0, \dotsc, \xi_d)$ of $n \in \Z_{\ge 0}$ and every $\kappa_0, \dotsc, \kappa_d \in \C$, we have
\begin{equation} \label{eq:dim.fil}
\sum_{(\pmb b; \pmb c) \in \mathcal I(\xi)} \prod_{j=0}^{|\pmb b(\pmb c(\xi))|-1} \dim L_{\lie{sl}} (\pmb b(\pmb c(\xi))_j)
= \prod_{j=0}^d \dim K(\kappa_j, \xi_j).
\end{equation}
\end{lemma}
\begin{proof}
First notice that, if $n = 0$, then $\xi = (0)$, $\mathcal I(\xi) = \{ \emptyset \}$ and both sides of equation~\eqref{eq:dim.fil} are equal to $1$.

Now we let $n > 0$. We proceed by induction on $n_\xi:=\#\{\xi_k \mid \xi_k\neq 0, \  0\leq k \leq d\}$. Suppose that $n_\xi = 1$ and $\xi_0 = 1$, that is, $\xi = (1, \dotsc, 1)$ and $d = n-1$.  In this case, the left side of equation~\eqref{eq:dim.fil} is
\[
\sum_{(b_1, \dotsc, b_k; c_1, \dotsc, c_\ell) \in \mathcal I(\xi)} \prod_{j=0}^{d - \ell - k} \dim L_{\lie{sl}} (1)
= \sum_{(b_1, \dotsc, b_k; c_1, \dotsc, c_\ell) \in \mathcal I(\xi)} 2^{d - \ell - k + 1}
= 4^n,
\]
where the last equality follows from repeatedly applying \eqref{e:binomial}.  Using Proposition~\ref{prop:dim.k2.sl12} and the fact that $d+1 = n$, one sees that $4^n$ equals the right side of equation~\eqref{eq:dim.fil}, $\prod_{j=0}^d \dim K(\kappa_j, 1)$ for any $\kappa_0, \dotsc, \kappa_d \in \C$.

Now, let $\xi = (\xi_0, \dotsc, \xi_0)$ for some $\xi_0 > 1$.  Notice that, for all $\pmb c = (c_1, \dotsc, c_\ell) \in \mathcal J (\xi)$, we have $\pmb c (\xi) = (\xi_0, \dotsc, \xi_0, \xi_0 - 1, \dotsc, \xi_0 - 1)$, which we denote by $(\underline{\xi_0}^{d+1-\ell}, \underline{\xi_0-1}^\ell)$.  In particular, if $\ell = d+1$, then $\pmb c = (0, 1, \dotsc, d)$ and $\pmb c (\xi) = (\xi_0 - 1, \dotsc, \xi_0-1)$.  Otherwise, by equation~\eqref{eq:dec.J}, $\mathcal J(\pmb c(\xi)) = \{ \pmb b^{\pmb u} \mid \pmb b \in \mathcal J(\underline{\xi_0}^{d+1-\ell}), \ \pmb u \in \mathcal J (\underline{\xi_0-1}^\ell) \}$ and
\[
\mathcal I (\xi) = \{ (\pmb b^{\pmb u}; \pmb c) \mid \pmb c = (c_1, \dotsc, c_\ell) \in \mathcal J(\xi), \ \pmb b \in \mathcal J(\underline{\xi_0}^{d+1-\ell}), \ \pmb u \in \mathcal J(\underline{\xi_0-1}^\ell) \}.
\]
Moreover, for each $\pmb c = (c_1, \dotsc, c_\ell) \in \mathcal J(\xi)$, $\pmb b = (b_1, \dotsc, b_k) \in \mathcal J(\underline{\xi_0}^{d+1-\ell})$ and $\pmb u = (u_1, \dotsc, u_p) \in \mathcal J(\underline{\xi_0-1}^\ell)$, we have
\[
\pmb b^{\pmb u} \left(\pmb c (\xi)\right)
= \pmb b ^{\pmb u} (\underline{\xi_0}^{d+1-\ell}, \underline{\xi_0-1}^\ell)
= (\pmb b(\underline{\xi_0}^{d+1-\ell}), \pmb u (\underline{\xi_0-1}^\ell))
= (\underline{\xi_0}^{d+1-\ell-k}, \underline{\xi_0-1}^{\ell-p+k}, \underline{\xi_0-2}^p).
\]
(Notice that, in the case $\xi_0 = 2$, the term $\underline{\xi_0-2}^p$ is zero and thus can be ignored.  Since $\dim(L_{\lie{sl}}(0)) = 1$, the corresponding term in the sequence of equations bellow can also be ignored.)  Hence, the left side of equation~\eqref{eq:dim.fil} is
\begin{align*}
\sum_{(c_1, \dotsc, c_\ell) \in \mathcal J(\xi)} {}&{} \sum_{(b_1, \dotsc, b_k) \in \mathcal J(\underline{\xi_0}^{d+1-\ell})} \, \sum_{(u_1, \dotsc, u_p) \in \mathcal J(\underline{\xi_0-1}^p)} \, \prod_{j=0}^{d} \dim L_{\lie{sl}} (\pmb b^{\pmb u} (\pmb c(\xi))_j) \\
{}&{}= \sum_{\ell = 0}^{d+1} \binom{d+1}{\ell} \sum_{k=0}^{d+1-\ell} \binom{d+1-\ell}{k} \sum_{p=0}^{\ell} \binom{\ell}{p} (\xi_0+1)^{d+1-\ell-k} \ \xi_0^{\ell-p+k} \ (\xi_0-1)^p \\
{}&{}= (4\xi_0)^{d+1},
\end{align*}
where the last equality follows from repeatedly using \eqref{e:binomial}.  Using Proposition~\ref{prop:dim.k2.sl12}, one can check that $(4\xi_0)^{d+1}$ equals $\prod_{j=0}^d \dim K(\kappa_j, \xi_0)$ for any $\kappa_0, \dotsc, \kappa_d \in \C$, the right side of equation~\eqref{eq:dim.fil}.
 
For the inductive step we assume that $n_\xi >1$. Thus $|\xi| > 1$ and there exists $t \in \{1, \dotsc, d\}$ such that $\xi_{t-1} > \xi_t$.  By Lemma~\ref{lem:dec.I}, the left side of equation~\eqref{eq:dim.fil} is
\begin{align*}
\sum_{(\pmb b^{\pmb u}; \pmb c^{\pmb z}) \in \mathcal I(\xi)} {}&{} \prod_{j=0}^{|\pmb b^{\pmb u}(\pmb c^{\pmb z}(\xi))|-1} \dim L_{\lie{sl}} (\pmb b^{\pmb u}(\pmb c^{\pmb z}(\xi))_j) \\
{}&{}= \sum_{\pmb c \in \mathcal J(\xi^\sharp_t)} \sum_{\pmb b \in \mathcal J(\pmb c(\xi^\sharp_t))} \sum_{\pmb z \in \mathcal J(\xi^\flat_t)} \sum_{\pmb u \in \mathcal J(\pmb z (\xi^\flat_t))} \prod_{j=0}^{t-1} \dim L_{\lie{sl}} (\pmb b(\pmb c(\xi^\sharp_t))_j) \prod_{j=t}^{d} \dim L_{\lie{sl}} (\pmb u(\pmb z(\xi^\flat_t))_j) \\
{}&{}= \left( \sum_{\pmb c \in \mathcal J(\xi^\sharp_t)} \, \sum_{\pmb b \in \mathcal J(\pmb c(\xi^\sharp_t))} \, \prod_{j=0}^{t-1} \dim L_{\lie{sl}} (\pmb b(\pmb c(\xi^\sharp_t))_j) \right) \left( \sum_{\pmb z \in \mathcal J(\xi^\flat_t)} \, \sum_{\pmb u \in \mathcal J(\pmb z (\xi^\flat_t))} \, \prod_{j=t}^{d} \dim L_{\lie{sl}} (\pmb u(\pmb z(\xi^\flat_t))_j) \right) \\
{}&{}= \left( \sum_{(\pmb b; \pmb c) \in \mathcal I(\xi^\sharp_t)} \, \prod_{j=0}^{t-1} \dim L_{\lie{sl}} (\pmb b(\pmb c(\xi^\sharp_t))_j) \right) \left( \sum_{(\pmb u; \pmb z) \in \mathcal I(\xi^\flat_t)} \, \prod_{j=t}^{d} \dim L_{\lie{sl}} (\pmb u(\pmb z(\xi^\flat_t))_j) \right),
\end{align*}
which equals $\left(\prod_{j=0}^{t-1} \dim K(\kappa_j, \xi_j)\right) \left(\prod_{j=t}^d \dim K(\kappa_j, \xi_j)\right)$ for any $\kappa_0, \dotsc, \kappa_d \in \C$, by induction hypothesis.
\end{proof}

The next lemma is used as an inductive step in the proof of Theorem~\ref{thm:iso.vcsi}.  But first, we remark a fact that is used in the proof of the lemma.

\begin{remark} \label{rem:exch.cond}
For every $\xi = (\xi_0, \dotsc, \xi_d)$ and $c \in \{0, \dotsc, d\}$, notice that $\varphi_c(\xi) = \varphi_{\hat c} (\xi)$ for some $\hat c \ge c$, namely $\hat c = \min \{ c' \mid c'\ge c, \ \xi_{c'} > \xi_{c'+1} \}$.  Moreover, notice that, when $\xi_c = \xi_{c+1} = \dotsb = \xi_{c+i}$, then there exists no $r > 0$ such that $\xi_{c+j+1} \le r < \xi_{c+j}$, $j \in \{0, \dotsc, i-1\}$.
\end{remark}

\begin{lemma} \label{lem:def.rel.red}
Let $\lambda \in P^+$, $\eta = (\eta_0, \dotsc, \eta_e)$ be a partition of $\lambda_2$, $\pmb \eta := (\lambda_1, \eta)$, $c \in \{0, \dotsc, e\}$, and $V$ be a $\g[t]$-module.  If $v \in V$ satisfies the relations \eqref{eq:defn.rels.wpsi} of $W(\lambda)$, the extra relations \eqref{eq:def.rel.vcsi} of $V(\pmb\eta)$, and
\[
(x_2 \otimes t)^s y_2^{r+s} y_3(c) v = 0
\quad \textup{for all} \quad
s = kr + \eta_{k+1} + \dotsb + \eta_e, \quad \eta_{k+1} \le r < \eta_k, \quad 0 < r, \quad 0 \le k \le c-1,
\]
then $y_3(c) v$ satisfies the extra relations \eqref{eq:def.rel.vcsi} of $V(\lambda_1-1, \varphi_c(\eta))$.
\end{lemma}
\begin{proof}
By Lemma~\ref{lem:vcsi.equi.defn}\ref{lem:vcsi.equi.defn.a}, if $v$ satisfies the relations \eqref{eq:defn.rels.wpsi} of $W(\lambda)$ and the extra relations \eqref{eq:def.rel.vcsi} of $V(\pmb\eta)$, then
\[
(x_2 \otimes t)^s y_2^{r+s} v = 0
\quad \textup{for all} \quad
s \ge kr + \eta_{k+1} + \dotsb + \eta_e + 1, \quad \eta_{k+1} \le r < \eta_k, \quad 0 < r, \quad 0 \le k \le e.
\]
Since $[y_3, y_2] = 0$, for all $c \ge 0$ we have:
\begin{equation} \label{eq:relsV(eta)}
(x_2 \otimes t)^s y_2^{r+s} y_3(c) v = 0
\quad \textup{for all} \quad
s \ge kr + \eta_{k+1} + \dotsb + \eta_e + 1, \ \eta_{k+1} \le r < \eta_k, \ 0 < r, \ 0 \le k \le e.
\end{equation}
Now, notice that, if $\eta_c = \dotsb = \eta_{c+i}$, $c+i = \hat c$ as in Remark~\ref{rem:exch.cond}, then: $\varphi_c(\eta) = \varphi_{\hat c}(\eta)$, $\eta_{\hat c} > \eta_{\hat c+1}$, and there exists no $r$ such that $\eta_{\hat c} \le r < \eta_c$.  Hence, $y_3(c)v$ satisfies the extra relations \eqref{eq:def.rel.vcsi} of $V(\lambda_1-1, \varphi_c(\eta))$ if, and only if,
\begin{equation}  \label{eq:relsV(phi(eta))}
(x_2 \otimes t)^s y_2^{r+s} y_3(c) v = 0
\quad \textup{for all} \quad
\begin{cases}
s \ge kr + \eta_{k+1} + \dotsb + \eta_e, & \eta_{k+1} \le r < \eta_k, \quad 0 \le k \le c-1,\\
s \ge lr + \eta_{l+1} + \dotsb + \eta_e + 1, & \eta_{l+1} \le r < \eta_l, \quad \hat c \le l \le e.
\end{cases}
\end{equation}
The result follows by comparing the relations in \eqref{eq:relsV(eta)} with those in \eqref{eq:relsV(phi(eta))}.
\end{proof}

\begin{remark} \label{rem:def.rel.red}
A result analogous to Lemma~\ref{lem:def.rel.red} remains valid if we replace $y_3$ by $x_1$.  In fact, the arguments of the proof are the same.
\end{remark}

Recall the filtration \eqref{eq:the.filtration}, defined in the proof of Proposition~\ref{prop:Weyl_basis}.  We close this subsection with a remark and a lemma regarding this filtration, that will also be used in the proof of Theorem~\ref{thm:iso.vcsi}.

\begin{remark}\label{r:quofilt}
For every $\lambda \in P^+$ and every quotient of $\g_0[t]$-modules $W(\lambda) \twoheadrightarrow V$, the image of the filtration \eqref{eq:the.filtration} $\{F(r)\}_{r \in \Z}$ under this quotient defines a filtration of $V$, as a $\g_0[t]$-module.  More explicitly, if $v \in V$ denotes the image of $w_\lambda$ through the quotient $W(\lambda) \twoheadrightarrow V$, then the image of $F(r)$ is given by $\sum_{i<r} \U(\g_0[t]) m_i v$ for all $r \in \Z$.
\end{remark}

\begin{lemma} \label{lem:filt.clsd}
Let $m_r = (b_1,\cdots, b_k; c_1,\cdots, c_\ell) \in \mathcal F$.  If $r', r'' \in \{0, \dotsc, N\}$ are such that $m_{r'} = m(b_1, \dotsc, b_k; c_2, \dotsc, c_\ell)$ and $m_{r''} = m(b_2, \dotsc, b_k; c_1, \dotsc, c_\ell)$, then:
\[
y_1(c_1) F(r') \subseteq F(r)
\qquad \textup{ and } \qquad
x_3(b_1) F(r'') \subseteq F(r).
\] 
\end{lemma}
\begin{proof}
The first containment follows from the following commutation relations: $y_1 x_2 - x_2 y_1 = [y_1, x_2] = 0$, $y_1 h - h y_1 = [y_1, h] = \alpha_1(h) y_1$ for all $h \in \h$, $y_1 y_2 - y_2 y_1 = [y_1, y_2] = - y_3$, $y_2 y_3 - y_3 y_2 = [y_2, y_3] = 0$, $y_1 x_1 + x_1 y_1 = [y_1, x_1] = h_1$, $h_1 x_1 - x_1 h_1 = [h_1, x_1] = 0$, $y_3 x_1 + x_1 y_3 = [y_3, x_1] = y_2$, $x_1 y_2 - y_2 x_1 = [x_1, y_2] = 0$, $y_1 y_3 + y_3 y_1 = [y_1, y_3] = 0$, and $y_2 h_1 - h_1 y_2 = [y_2, h_1] = -y_2$.

The second containment follows from the following commutation relations: $x_3 x_2 - x_2 x_3 = [x_3, x_2] = 0$, $x_3 h - h x_3 = [x_3, h] = -\alpha_3(h) x_3$ for all $h \in \h$,
$x_3 y_2 - y_2 x_3 = [x_3, y_2] = x_1$, $x_1 y_2 - y_2 x_1 = [x_1, y_2] = 0$, $x_3 x_1 + x_1 x_3 = [x_3, x_1] = 0$, $x_3 y_3 + y_3 x_3 = [x_3, y_3] = h_3$, $h_3 y_3 - y_3 h_3 = [h_3, y_3] = 0$, $x_1 h_3 - h_3 x_1 = [z_1, h_3] = -x_1$ and $x_2 h_3 - h_3 x_2 = [x_2, h_3] = x_2$.
\end{proof}

\subsection{The structure of Chari-Venkatesh modules}

In this subsection we describe the structure of Chari-Venkatesh modules.  We begin with a result relating them to fusion products.

\begin{theorem} \label{thm:iso.vcsi}
Let $\lambda = (\lambda_1, \lambda_2) \in P^+$, $\xi = (\xi_0, \dotsc, \xi_d)$ be a partition of $\lambda_2$, and $\pmb\xi := (\lambda_1, \xi)$.  For every choice of pairwise-distinct complex numbers $z_0, \dotsc, z_d$ and complex numbers $\kappa_0, \dotsc \kappa_d$ satisfying $\kappa_0 + \dotsb + \kappa_d = \lambda_1$, there is an isomorphism of $\g[t]$-modules
\[
V (\pmb\xi) \cong K(\kappa_0, \xi_0)^{z_0} * \dotsb * K(\kappa_d, \xi_d)^{z_d}.
\]
In particular, $\dim V (\pmb\xi) = 4^{d+1} \xi_0 \dotsm \xi_d$.
\end{theorem}
\begin{proof}
First recall from Proposition~\ref{prop:vcsi.onto.kacs} that there is an epimorphism of $\g[t]$-modules $V(\pmb\xi) \twoheadrightarrow K(\kappa_0, \xi_0)^{z_0} * \dotsb * K(\kappa_d, \xi_d)^{z_d}$, and in particular, that $\dim V (\pmb\xi) \ge 4^{d+1} \xi_0 \dotsm \xi_d$, by Proposition~\ref{prop:dim.k2.sl12}.  Thus, it is sufficient to prove that $\dim V (\pmb\xi) \le 4^{d+1} \xi_0 \dotsm \xi_d$.

In order to do that, recall that $V(\pmb\xi)$ is, by definition, a quotient of the local Weyl $\U(\g[t])$-module $W(\lambda)$.  Hence, by Remark \ref{r:quofilt} we have that the filtration $\{F(r)\}_{r \in \Z}$, constructed in the proof of Proposition~\ref{prop:Weyl_basis}, induces a filtration on $V(\pmb \xi)$.  Namely, for each $r \in \Z$, the image of $F(r)$ under the quotient $W(\lambda) \twoheadrightarrow V(\pmb \xi)$ is
\[
E(r) := \sum_{i<r} \U(\g_0[t])v_i, \quad \textup{where} \quad  v_i := m_i v_{\pmb\xi}.
\]
Denote by ${\rm Gr}(V(\pmb\xi))$ the associated graded module. Recall from the proof of Proposition~\ref{prop:Weyl_basis} that, for each $r \in \{0, \dotsc, N\}$, $v_r + E(r) \in {\rm Gr}(V(\pmb\xi))$ satisfies the defining relations of a highest-weight generator of the local Weyl $\U(\lie{sl}(2)[t])$-module of highest weight $\lambda_2 - \wt(m_r)$.  We want to prove that, in fact, for each $r \in \{0, \dotsc, N\}$, $v_r + E(r) \in {\rm Gr}(V(\pmb\xi))$ also satisfies the extra relations of a certain Chari-Venkatesh $\lie{sl}(2)[t]$-module.  (Here, and throughout this proof, we identify $\g'_0$ with $\mathfrak{sl}(2)$ and view each $E(r)$ as and $\mathfrak{sl}(2)[t]$-module via restriction.)

Let $i \in \{0, \dotsc, N\}$, denote $m_i = m(\pmb b; \pmb c) = x_1(b_1) \dotsb x_1(b_k) y_3(c_1) \dotsb y_3(c_\ell)$, and define $\pmb b (\pmb c(\pmb \xi)) := \left( \lambda_1 - \wt(m_i), \, \pmb b(\pmb c(\xi)) \right)$.  We prove by induction on $i$ that $v_i + E(i)$ satisfies the extra relations \eqref{eq:def.rel.vcsi} of $V(\pmb b(\pmb c(\pmb\xi)))$.  For $i=0$, we have $m_0 = 1$ and $v_0 = v_{\pmb\xi}$, which satisfies the extra relations \eqref{eq:def.rel.vcsi} of $V(\pmb\xi)$ by construction.  Now, assume that $i>0$.

Suppose first that $\pmb b = \emptyset$, denote $\varphi_{c_2} \circ \dotsb\circ \varphi_{c_\ell} (\xi) =: \eta = (\eta_0, \dotsc, \eta_e)$ and $\pmb\eta := (\lambda_1-\ell+1, \, \eta)$.  We want to prove that $v_i + E(i) \in {\rm Gr}(V(\pmb\xi))$ satisfies the extra relations \eqref{eq:def.rel.vcsi} of $V(\lambda_1-\ell, \, \varphi_{c_1}(\eta))$.  In order to do that, let $m = m(\pmb b; c_2,\cdots, c_\ell)$ and notice that $m \prec m_i$, as $m_i = y_3(c_1)m$.  Hence $m = m_j$ for some $j < i$, and by induction hypothesis, we have that $v_j + E(j)$ satisfies the extra relations of $V(\pmb\eta)$, that is,
\begin{equation} \label{eq:rels.veta1}
\mathbf y_2(r,s) v_j \in E(j)
\quad \textup{for all} \quad s > kr + \eta_{k+1} + \dotsb + \eta_e, \ 
\eta_{k+1}\leq r <\eta_k, \ 0 < r, \ 0 \leq  k \leq e.
\end{equation}
Since $y_3y_2 - y_2y_3 = [y_3, y_2] = 0$ and $E(j) \subseteq E(i)$, by Lemma~\ref{lem:def.rel.red}, it suffices to show that
\begin{equation}\label{e:IH}
\mathbf y_2(r,s) y_3(c_1) v_j \in E(i)
\ \textup{for all} \ 
s = kr + \eta_{k+1} + \dotsb + \eta_e, \ \eta_{k+1} \le r < \eta_k, \ 0 < r, \ 0 \le k \le c_1-1.
\end{equation}
For this let $0\leq k\leq c_1-1$, $\eta_{k+1}\leq r< \eta_k$ and $s=kr +\eta+{k+1}+\cdots + \eta_e$.  Notice that $s+c_1 > k(r+1)+\eta_{k+1}+\dotsb+\eta_e$ when $r+1 < \eta_k$, and $s+c_1 > (k-1)(r+1) + \eta_k + \eta_{k+1}+ \dotsb + \eta_e$ when $r+1 = \eta_k$.  Moreover if $k = 0$ and $r+1 = \eta_0$, then $\mathbf y_2(r+1,s+c_1)v_j = 0$ by Lemma~\ref{lem:vcsi.equi.defn}\ref{lem:vcsi.equi.defn.a}.  Either way, \eqref{eq:rels.veta1} implies that $\mathbf y_2 (r + 1, s + c_1)v_j \in E(j)$.  Hence, by \eqref{eq:com.rel8} and Lemma~\ref{lem:filt.clsd}, we have:
\[
\sum_{q=0}^{s + c_1} \mathbf y_2 (r, s + c_1 - q) y_3(c_1+ q) v_j
= - y_1(c_1)\mathbf y_2 (r + 1, s + c_1) v_j \in E(i).
\]
By definition of $E(r)$ we have that $\U(\lie g_0[t])y_3(c_1+q)v_j \in E(i)$ if $q \ne 0$ and hence we obtain that $\mathbf y_2 (r, s) y_3(c_1)v_j \in E(i)$, as we wanted.
 
For $\pmb b \ne \emptyset$, denote $\varphi_{b_2} \circ \dotsb \circ \varphi_{b_k} \circ \varphi_{c_1} \circ \dotsb \circ \varphi_{c_\ell} (\xi) =: \eta = (\eta_1, \dotsc, \eta_e)$ and $\pmb\eta := (\lambda_1 - k - \ell + 1, \, \eta)$. Our goal is to prove that $v_i + E(i)$ satisfies the extra relations \eqref{eq:def.rel.vcsi} of $V(\lambda_1 - k - \ell, \, \varphi_{b_1}(\eta))$.  This proof is analogous to the one in the case $\pmb b = \emptyset$.  It is obtained by exchanging $x_1$ for $y_3$ and $c_1$ for $b_1$ everywhere in the above paragraph, then using Remark~\ref{rem:def.rel.red} instead of Lemma~\ref{lem:def.rel.red} and equation~\eqref{eq:com.rel9} instead of \eqref{eq:com.rel8}.

The above argument shows that, for all $i \in \{0, \dotsc, N\}$, the $\U(\g_0[t])$-submodule generated by $m_i v_{\pmb\xi} + E(i) \in {\rm Gr}(V(\pmb\xi))$ is a quotient of a Chari-Venkatesh $\mathfrak{sl}(2)[t]$-module.  Namely, if $m_i = m(\pmb b; \pmb c)$, $(\pmb b; \pmb c) \in \mathcal I(\xi)$, then $\U(\g_0[t]) m_i v_{\pmb\xi} + E(i) \in {\rm Gr}(V(\pmb\xi))$ is a quotient of the Chari-Venkatesh $\mathfrak{sl}(2)[t]$-module associated to the partition $\pmb b(\pmb c(\xi))$.  Moreover, in case $m(\pmb b; \pmb c) \notin \mathcal F$ for some $(\pmb b; \pmb c) \in \mathcal I(\xi)$, we have that $m(\pmb b; \pmb c) v_\xi = 0$.  In this case, the $\U(\g_0[t])$-submodule generated by the image of $m(\pmb b; \pmb c) v_\xi$ in ${\rm Gr}(V(\pmb\xi))$ is also a quotient of the Chari-Venkatesh $\mathfrak{sl}(2)[t]$-module associated to the partition $\pmb b(\pmb c(\xi))$.  Together with \cite[Theorem~5(ii)]{CV15} and Lemma~\ref{lem:dim.fil}, this implies that
\[
\dim V(\pmb\xi)
\leq \sum_{(\pmb b; \pmb c) \in \mathcal I(\xi)} \prod_{j=0}^{|\pmb b(\pmb c(\xi))|-1} \dim L_{\lie{sl}} (\pmb b(\pmb c(\xi))_j)
= 4^{d+1} (\xi_0 \dotsm \xi_d).
\]
The result follows.
\end{proof}

\begin{remark}
Recall from Proposition~\ref{prop:typ.cond} that the subset $\{ \mu \in P^+ \mid \mu \textup{ is typical} \}$ is Zariski open in $P^+$.  Hence, for every $\lambda \in P^+$ and every partition $\xi = (\xi_0, \dotsc, \xi_d)$ of $\lambda_2$, one can choose $\kappa_0, \dotsc, \kappa_d \in \C$ such that $\kappa_0 + \dotsb + \kappa_d = \lambda_1$ and $(\kappa_i, \xi_i) \in P^+$ is typical for all $i \in \{0, \dotsc, d\}$ (or equivalently, we may assume that $K(\kappa_i, \xi_i)$ is irreducible for every $i \in \{0, \dotsc, d\}$). Thus, Theorem~\ref{thm:iso.vcsi} generalizes \cite[Theorem~5(ii)]{CV15}, except in the case $|\xi| = 1$ (where $V(\pmb\xi) \cong {\rm ev}_0 K(\lambda)$, see Lemma~\ref{lem:vcsi.basc}\ref{lem:vcsi.basc.a}).
\end{remark}

\begin{corollary}\label{cor:iso.vcsi}
Let $\lambda = (\lambda_1, \lambda_2) \in P^+$, $\xi = (\xi_0, \dotsc, \xi_d)$ be a partition of $\lambda_2$, and $\pmb\xi := (\lambda_1, \xi)$. For every choice of pairwise-distinct complex numbers $z_0, \dotsc, z_d$ and complex numbers $\kappa_0, \dotsc \kappa_d$ satisfying $\kappa_0 + \dotsb + \kappa_d = \lambda_1$, the following statements hold:

\begin{enumerate}[(a), leftmargin=*, itemsep=1ex]
\item \label{cor:CV.indep}
The fusion product $K(\kappa_0, \xi_0)^{z_0} * \dotsb * K(\kappa_d, \xi_d)^{z_d}$ does not depend on $z_0, \dotsc, z_n$, nor on $\kappa_0, \dotsc \kappa_d$, provided that $\kappa_0 + \dotsb + \kappa_d = \lambda_1$.

\item \label{cor:CV.filt}
As a $\g_0[t]$-module, $V(\pmb \xi)$ admits a filtration, namely $\{E(r)\}_{r \in \Z}$, whose factors are the Chari-Venkatesh modules associated to the partitions $\pmb b(\pmb c(\xi))$, $(\pmb b; \pmb c) \in \mathcal I(\xi)$.

\item \label{cor:CV.char}
The character of $V(\pmb\xi)$ is given by
	\[
\ch V (\pmb\xi) = e^{(\lambda_1, 0)} \frac{\left( \ch \Lambda(\g_{-1})\right)^{d+1}}{\left( e^{(0,\, 1)} - e^{(0,\, -1)} \right)^{d+1}} \prod_{i=0}^{d} \left( e^{(0,\, \xi_i-1)} - e^{(0,\, 1-\xi_i)} \right).
	\]

\item \label{cor:CV.basis}
The Chari-Venkatesh module $V(\pmb \xi)$ admits a basis consisting of the following elements:
\[ y_2(0)^{(i_1)} \dotsm y_2(e-1)^{(i_e)} x_1(b_1) \dotsm x_1(b_k) y_3(c_1) \dotsm y_3(c_\ell) v_{\pmb \xi}, \]
where $(\pmb b; \pmb c) := (b_1, \dotsc, b_k; c_1, \dotsc, c_\ell) \in \mathcal I(\xi)$, $e := |\pmb b(\pmb c(\xi))|$ and $(i_1, \dotsc, i_e) \in \Z_{\ge0}^e$ are such that $p i_{q-1} + (p+1) i_q + 2 \sum_{t=q+1}^e i_t \le \sum_{t=q-p}^e \pmb b(\pmb c(\xi))_t$ for all $q \in \{2, \dotsc, e+1\}$, $p \in \{1, \dotsc, q-1\}$.
\end{enumerate}
\end{corollary}
\begin{proof}
Parts \ref{cor:CV.indep} and \ref{cor:CV.filt} follow directly from Theorem~\ref{thm:iso.vcsi} and its proof.  In order to prove part \ref{cor:CV.char}, first notice that Theorem~\ref{thm:iso.vcsi} implies that $\ch V(\pmb \xi) = \ch K(\kappa_0, \xi_0) \dotsm \ch K(\kappa_d, \xi_d)$.  Then, recall from Lemma~\ref{lem:kac.iso} that, for each $i \in \{0, \dotsc, d\}$, either $\ch K(\kappa_i, \xi_i) = \ch K_{{\lie b}_{\scriptscriptstyle (1)}}(\kappa_i,\, \xi_i-1)$, or $\ch K(\kappa_i, \xi_i) = \ch K_{{\lie b}_{\scriptscriptstyle (3)}}(\kappa_i-1, \, \xi_i-1)$. To finish the proof of this statement, notice that, for each $i \in \{0, \dotsc, d\}$,
\begin{align*}
\ch K_{{\lie b}_{\scriptscriptstyle (1)}}(\kappa_i, \xi_i-1)
& = \ch L_{\gl}(\kappa_i, \xi_i-1) \ch \Lambda(\g_{-1}) \\
& =  \ch L_{\gl}(\kappa_i-1, \xi_i-1) \ch \Lambda(\g_{1}) 
= \ch K_{{\lie b}_{\scriptscriptstyle (3)}}(\kappa_i-1, \xi_i-1).
\end{align*}
Part~\ref{cor:CV.basis} follows from part~\ref{cor:CV.filt} and \cite[Theorem~5(iii)]{CV15}.
\end{proof}

\begin{remark}\label{rem:fus.prod.typical}
It follows from Corollary~\ref{cor:iso.vcsi}\ref{cor:CV.indep} that  fusion products of finite-dimensional typical irreducible modules do not depend on the choice of the evaluation points $z_0, \dotsc, z_d$.  However, we cannot realize fusion products of finite-dimensional atypical irreducible modules as Chari-Venkatesh modules (and hence as local Weyl modules), as the dimensions would not match.
\end{remark}

\begin{remark}
As a particular case of Corollary~\ref{cor:iso.vcsi}\ref{cor:CV.filt}, the $\g_0'[t]$-submodule generated by $v_{\pmb \xi}$ is isomorphic to the Chari-Venkatesh $\lie{sl}(2)[t]$-module associated to the partition $\xi$.
\end{remark}

\section{Demazure-type and truncated local Weyl modules} \label{s:dem.tr.weyl}

In this last section, we apply the results obtained above to two classes of $\g[t]$-modules which have important counter-parts in the non-super setting.

\subsection{Demazure-type modules}

In \cite{Kus18}, D.~Kus defined analogues of Demazure modules for the Lie superalgebra $\lie{osp}(1|2)[t]$, and called them \emph{Demazure-type modules}.  In this subsection, we give a similar definition for the Lie superalgebra $\lie{sl}(1|2)[t]$.

\begin{definition}
Given $\ell\in \mathbb Z_{\geq 0}$ and $\lambda = (\lambda_1, \lambda_2) \in P^+$, define the \emph{Demazure-type module} $D(\ell, \lambda)$ to be the quotient of $\U(\lie g[t])$ by the left ideal generated by  
\begin{equation} \label{eq:def.rel.dem}
\lie n^+[t], \quad
h(r) - \lambda(h)\delta_{r,0}, \quad
y_2(r)^{\max\{0, \lambda_2-\ell r\}+1}, \qquad 
\textup{for all $h \in \h$, $r \in \Z_{\ge0}$}.
\end{equation}
Notice that $D(\ell, \lambda)$ is generated by the image of $1 \in \U(\g[t])$ onto $D(\ell, \lambda)$, which is a homogeneous even vector and which we denote by $v_{\ell, \lambda}$.
\end{definition}

The next result is a characterization of Demazure-type modules in terms of fusion products of generalized Kac--modules which will be key to following results of this subsection.

\begin{proposition} \label{prop:dem.vcsi}
Let $\ell \in \Z_{>0}$ and $\lambda = (\lambda_1, \lambda_2) \in P^+$.  Let $\lambda_2 = (q-1)\ell + m$, $0 < m \le \ell$.  For every $\kappa_1, \dotsc, \kappa_q \in \C$ such that $\kappa_1 + \dotsb + \kappa_q = \lambda_1$ and pairwise-distinct $z_1, \dotsc, z_q \in \C$, we have an isomorphism of $\lie g[t]$-modules:
\[
D(\ell, \lambda) \cong K(\kappa_1, \ell)^{z_1} * \dotsm * K(\kappa_{q-1}, \ell)^{z_{q-1}} * K(\kappa_q, m)^{z_q}.
\]
\end{proposition}
\begin{proof}
By Theorem \ref{thm:iso.vcsi} it is equivalent to proving that we have an isomorphism of $\lie g[t]$--modules $D(\ell,\lambda)\cong V(\pmb \xi)$, with $\pmb\xi = (\lambda_1,\xi)$, $\xi = (\underline{\ell}^{q-1}, m)$. The proof follows the same lines as the proof of \cite[Proposition~5.2]{Kus18}.
\end{proof}

The following formulas for dimensions and characters of Demazure-type modules are consequences of the isomorphism given in Proposition~\ref{prop:dem.vcsi}.

\begin{corollary} \label{cor:app.dem}
Let $\ell \in \Z_{>0}$, $\lambda = (\lambda_1, \lambda_2) \in P^+$, and $\lambda_2 = (q-1)\ell + m$, $0 < m \le \ell$.
\begin{enumerate}[(a), leftmargin=*]
\item \label{cor:app.dem.a}
If $\ell = 1$, then $D(\ell, \lambda) \cong W(\lambda)$.

\item \label{cor:app.dem.b}
$\dim D(\ell, \lambda) = 4^q \ell^{q-1} m$.

\item \label{cor:app.dem.c}
$\ch D(\ell, \lambda) = e^{(\lambda_1, 0)} \frac{\left( \ch \Lambda(\g_{-1}) \right)^q}{\left( e^{(0,1)} - e^{(0,-1)} \right)^q} \left( e^{(0,\, \ell-1)} - e^{(0,\, 1-\ell)} \right)^{q-1} \left( e^{(0,\, m-1)} - e^{(0,\, 1-m)} \right)$.
\end{enumerate}
\end{corollary}
\begin{proof}
Part~\ref{cor:app.dem.a} follows directly from Proposition~\ref{prop:dem.vcsi} and Theorem~\ref{thm:weyl.dim}.  Part~\ref{cor:app.dem.b} follows directly from Proposition~\ref{prop:dem.vcsi} and Lemma~\ref{lem:dim.wpsi}\ref{lem:dim.wpsi.b}.  The proof of part~\ref{cor:app.dem.c} is a straight-forward computation using Proposition~\ref{prop:dem.vcsi}, Proposition~\ref{prop:kac.g0-mods} and the following character formulas for finite-dimensional irreducible $\lie{gl}(2)$-modules:
\[
\ch L_{\lie{gl}} (\mu)
= e^{(\mu_1, 0)} \sum_{i = 0}^{\mu_2} e^{(0,\, \mu_2-2i)}
= e^{(\mu_1, 0)} \frac{\left( e^{(0,\, \mu_2+1)} - e^{(0,\, -\mu_2-1)} \right)}{\left( e^{(0,\, 1))} - e^{(0,\, -1)} \right)},
\qquad \mu = (\mu_1, \mu_2) \in P^+.
\qedhere
\]
\end{proof}

This next result characterizes these Demazure-type modules as lowest-weight modules.  This goes back to their original definition as modules for a positive Borel subgroup of an affine Lie algebra generated by an extremal weight vector.  Compare it with \cite[Lemme~26]{Mathieu89}, \cite[Theorem~1]{FL07} and \cite[Proposition~3.6]{Naoi12}.

\begin{proposition} \label{prop:dem.lw}
Let $\lambda = (\lambda_1, \lambda_2) \in P^+$, $\ell \in \Z_{\ge0}$, and set $\mu := (\lambda_1-\lambda_2, \ -\lambda_2) \in \h^*$.  The Demazure-type module $D(\ell, \lambda)$ is generated by a non-zero vector $v_-$ satisfying the following defining relations:
\begin{equation} \label{eq:alt.rel.dem}
\n^-[t] v_- = 0, \quad
h(r)v_- = \delta_{r,0} \mu(h)v_-, \quad
x_2(r)^{\max\{0, \lambda_2 - \ell r\}+1} v_- = 0, \quad\ 
\textup{for all $h \in \h$, $r \in \Z_{\ge0}$}.
\end{equation}
\end{proposition}
\begin{proof}
Let $v_- := y_2^{\lambda_2} v_{\ell, \lambda}$.  We firstly show that $v_- = y_2^{\lambda_2} v_{\ell, \lambda}$ is a generator of $D(\ell, \lambda)$.  Since $v_{\ell, \lambda}$ is a generator of the $\g'_0$-submodule $\U(\g'_0)v_{\ell, \lambda} \subseteq D(\ell, \lambda)$ and $\g'_0 \cong \lie{sl}(2)$, using representation theory of $\lie{sl}(2)$, one can see that $v_-$ is also a generator of $\U(\g'_0)v_{\ell, \lambda}$.  Furthermore, since $v_{\ell, \lambda}$ is a generator of $D(\ell, \lambda)$ and $v_{\ell, \lambda} \in \U(\g'_0)v_-$, then $v_-$ is also a generator of $D(\ell, \lambda)$.

Secondly, we show that $v_-$ satisfies the relations \eqref{eq:alt.rel.dem}.  Since $[y_3, y_2]= 0$, then $y_3(c) v_- = y_3(c) y_2^{\lambda_2} v_{\ell, \lambda} = y_2^{\lambda_2} y_3(c) v_{\ell, \lambda}$ for all $c \in \Z_{\ge0}$.  Now, using representation theory of $\lie{sl}(2)$ (cf. proof of Proposition~\ref{prop:Weyl_basis}), one can see that $y_2^{\lambda_2} y_3(c) v_{\ell, \lambda} = 0$ for all $c \in \Z_{\ge0}$.  Similarly, since $[x_1, y_2] = 0$, then $x_1(b) v_- =  x_1(b)y_2^{\lambda_2} v_- = y_2^{\lambda_2} x_1(b) v_{\ell, \lambda}$ for all $b \in \Z_{\ge0}$.  Now, using representation theory of $\lie{sl}(2)$ (cf. proof of Proposition~\ref{prop:Weyl_basis}), one can see that $y_2^{\lambda_2} x_1(b) v_{\ell, \lambda} = 0$ for all $b \in \Z_{\ge0}$.  Moreover, $y_2(a) v_- = y_2(a) y_2^{\lambda_2} v_{\ell, \lambda}$, and using representation theory of $\lie{sl}(2)[t]$ (cf. \cite[Section~6]{CP01}), one can see that $y_2(a) y_2^{\lambda_2} v_{\ell, \lambda} = 0$ for all $a \in \Z_{\ge0}$.  This shows that $\n^-[t] v_- = 0$.

Next, we verify that $h(r)v_- = \delta_{r,0} \mu(h)v_-$ for all $h \in \h$ and $r \in \Z_{\ge0}$.  Since $[h_1, y_2] = -y_2$, then $h_1(r) v_- = h_1(r) y_2^{\lambda_2} v_{\ell, \lambda} = y_2^{\lambda_2} h_1(r) v_{\ell, \lambda} - \lambda_2 y_2(r) y_2^{\lambda_2-1} v_{\ell, \lambda}$.  Now, using representation theory of $\lie{sl}(2)[t]$, we see that $y_2^{\lambda_2} h_1(r) v_{\ell, \lambda} - \lambda_2 y_2(r) y_2^{\lambda_2-1} v_{\ell, \lambda} = \delta_{0,r} (\lambda_1-\lambda_2) y_2^{\lambda_2} v_{\ell, \lambda} = \delta_{0,r} \mu(h_1) v_-$ for all $r \in \Z_{\ge0}$.  Similarly, since $[h_2, y_2] = -2y_2$, then $h_2(r) v_- = h_2(r) y_2^{\lambda_2} v_{\ell, \lambda} = y_2^{\lambda_2} h_2(r) v_{\ell, \lambda} - 2\lambda_2 y_2(r) y_2^{\lambda_2-1} v_{\ell, \lambda}$.  Representation theory of $\lie{sl}(2)[t]$, implies that $y_2^{\lambda_2} h_2(r) v_{\ell, \lambda} - 2\lambda_2 y_2(r) y_2^{\lambda_2-1} v_{\ell, \lambda} = -\delta_{0,r}\lambda_2 y_2^{\lambda_2} v_{\ell, \lambda} = \delta_{0,r} \mu(h_2) v_-$ for all $r \in \Z_{\ge0}$.

Finally, we prove that $x_2(r)^{\max\{0, \lambda_2 - \ell r\}+1} v_- = 0$ for all $r \in \Z_{\ge0}$.  First, recall from Corollary~\ref{cor:iso.vcsi}\ref{cor:CV.filt} and Proposition~\ref{prop:dem.vcsi} that the $\g'_0[t]$-submodule $\U(\g'_0[t])v_{\ell, \lambda} \subseteq D(\ell, \lambda)$ is isomorphic to the Chari-Venkatesh $\lie{sl}(2)[t]$-module associated to the partition $(\underline \ell^{q-1}, \, m)$.  Hence, by \cite[Theorem~2]{CV15}, the $\g'_0[t]$-submodule $\U(\g'_0[t]) v_{\ell, \lambda}$ is isomorphic to the Demazure $\lie{sl}(2)[t]$-module associated to $(\ell, \lambda_2)$.  Now, the relation $x_2(r)^{\max\{0, \lambda_2 - \ell r\}+1} v_- = 0$ follows from \cite[Theorem~1]{FL07} (which is the analog of this proposition for simply-laced Lie algebras, in particular, $\lie{sl}(2)$).

Notice that, so far, we have proved that $D(\ell, \lambda)$ is a quotient of the cyclic $\g[t]$-module $D_-$ generated by a cyclic vector $v_-$ satisfying \eqref{eq:alt.rel.dem}.  To finish the proof, notice that a similar argument would show that $x_2^{\lambda_2} v_- \in D'$ satisfies the defining relations \eqref{eq:def.rel.dem} of $v_{\ell, \lambda}$, thus proving that $D_-$ is also a quotient of $D(\ell, \lambda)$.
\end{proof}

\subsection{Truncated local Weyl modules}

In this last subsection, we study analogs of the so-called truncated local Weyl modules, that have been defined for Lie algebras of the form $\lie l \otimes \C[t]/\langle t^n \rangle$ where $\lie l$ is a finite-dimensional simple Lie algebra and $n \in \Z_{>0}$ (see \cite{FMM19} and the references therein).  These modules have shown to be important for their applications, for instance, to a conjecture on Schur positivity of symmetric functions stated in \cite{CFS14}.

\begin{definition}
Given $\lambda \in P^+$ and $N\in \mathbb Z_{\geq 0}$, define the \emph{truncated local Weyl module} $W(\lambda, N)$ to be the quotient of $\U(\lie g[t])$ by the left ideal generated by 
\begin{equation} \label{eq:defn.rels.trweyl}
\lie n^+[t], \quad
h(r) - \lambda(h)\delta_{r,0}, \quad
y_2^{\lambda_2+1}, \quad
\lie g\otimes t^N\mathbb C[t], 
\qquad \textup{for all} \ h \in \lie h, \ r\in \mathbb Z_{\geq 0}.
\end{equation}
Notice that $W(\lambda, N)$ is generated by the image of $1 \in \U(\g[t])$ onto $W(\lambda, N)$, which is a homogeneous even vector and which we denote by $w_{\lambda,N}$.
\end{definition}

As in the Demazure-type modules we also can describe truncated local Weyl modules in terms of fusion products of generalized Kac modules.

\begin{proposition} \label{prop:tr.weyl.iso.fusion}
Let $\lambda = (\lambda_1,\lambda_2)\in P^+$ and $N\in \mathbb Z_{\geq 0}$. Let $\lambda_2 = qN + m$, $0 \le m < N$, be the Euclidean division of $\lambda_2$ by $N$.  For every $\kappa_1, \dotsc, \kappa_N \in \C$ such that $\kappa_1 + \dotsb + \kappa_N = \lambda_1$ and pairwise-distinct $z_1, \dotsc, z_N \in \C$, we have an isomorphism of $\lie g[t]$-modules:
\[
W(\lambda,N) \cong  K(\kappa_1,q)^{z_1} * \dotsm * K(\kappa_{N-m},q)^{z_{N-m}} * K(\kappa_{N-m+1}, q+1)^{z_{N-m+1}} * \dotsm * K(\kappa_N, q+1)^{z_N}.
\]
\end{proposition}
\begin{proof}
By Theorem \ref{thm:iso.vcsi}, it suffices to prove that $W(\lambda,N)$ is isomorphic to $V(\pmb\xi)$, for $\pmb\xi = (\lambda_1; \xi)$ and $\xi = (\underline{q+1}^m, \underline{q}^{N-m})$.  It is clear that $W(\lambda,N)$ is a quotient of $W(\lambda)$. Moreover, using Remark~\ref{rmk:vanishing.fusion}, one can check that $V(\pmb\xi)$ is a quotient of $W(\lambda,N)$.  Therefore, it suffices to show that $w_{\lambda,N}$ satisfies the extra relations \eqref{eq:def.rel.vcsi} of $V(\pmb\xi)$, for $\pmb\xi = (\lambda_1; \xi)$ and $\xi = (\underline{q+1}^m, \underline{q}^{N-m})$.  In light of Lemma~\ref{lem:vcsi.equi.defn}\ref{lem:vcsi.equi.defn.b}, it suffices to prove that, for every $r \le q$ and $s > (N-1)r$, we have $\mathbf y_2(r,s) v_{\pmb \xi} = 0$.  Now, using Lemma~\ref{lem:vcsi.equi.defn}\ref{lem:vcsi.equi.defn.c}, it suffices to prove that if $b_0 + \dotsb + b_s = r$, $r \le q$, $b_1 + 2b_2 + \dotsb + sb_s = s$ and $s > (N-1)r$, then there exists $p \ge N$ such that $b_p \ne 0$.  If $b_p = 0$ for all $p \ge N$, then
\[
(N-1)r< s = b_1 + 2b_2 + \cdots + (N-1)b_{N-1} \leq (N-1)r,
\]
as  $b_p \le r$ for all $p \in \{0, \dotsc, r\}$.  This is a contradiction, and thus the result follows.
\end{proof}

The following corollary is a consequence of the previous isomorphism and the proof is similar to that of Corollary~\ref{cor:app.dem}. 

\begin{corollary} \label{cor:app.tr.weyl}
Let $\lambda = (\lambda_1,\lambda_2)\in P^+$ and $N\in \mathbb Z_{\geq 0}$. Let $\lambda_2 = qN + m$, $0 \le m < N$, be the Euclidean division of $\lambda_2$ by $N$.
\begin{enumerate}[(a), leftmargin=3.75ex]
\item \label{cor:app.tr.weyl.a}
If $N \ge \lambda_2$, then $W(\lambda, N) \cong W(\lambda)$.

\item \label{cor:app.tr.weyl.b}
If $N < \lambda_2$, then $\dim W(\lambda, N) = 4^N q^{N-m} (q+1)^m$.

\item \label{cor:app.tr.weyl.c}
If $N < \lambda_2$, then $\ch W(\lambda, N) = e^{(\lambda_1, 0)} \frac{\left( \ch \Lambda (\g_{-1}) \right)^N}{\left( e^{(0,1)} - e^{(0,-1)} \right)^N} \left( e^{(0,\, q)} - e^{(0,\, -q)} \right)^{N-m} \left( e^{(0,\, q-1)} - e^{(0,\, 1-q)} \right)^m$.
\end{enumerate}
\end{corollary}

Notice that, since $\left( \g \otimes t^N\C[t] \right) W(\lambda, N) = 0$, then $W(\lambda, N)$ admits a structure of $\g \otimes \C[t]/\langle t^N \rangle$-module.  And conversely, every $\g \otimes \C[t]/\langle t^N \rangle$-module $V$ admits a structure of $\g[t]$-module where $\left( \g \otimes t^N\C[t] \right) V = 0$.  We close the paper with a result characterizing truncated local Weyl modules as universal objects within a certain category of $\g \otimes \C[t]/\langle t^N \rangle$-modules.  Compare it with \cite[Section~3]{CFK10}, \cite[Proposition~4.13]{CLS19} and \cite[Proposition~8.2]{BCM19}.

\begin{proposition} \label{prop:tr.weyl.univ}
Let $\lambda \in P^+$ and $N \in \Z_{\ge0}$.  If $V$ is a finite-dimensional graded highest-weight $\g \otimes \C[t] / \langle t^N \rangle$-module generated by a highest-weight vector $v_+$ of weight $\lambda$, then there exists a surjective homomorphism of $\g[t]$-modules:
\[
W(\lambda, N) \twoheadrightarrow V
\qquad \textup{satisfying} \qquad
w_{\lambda, N} \mapsto v_+.
\]
\end{proposition}
\begin{proof}
This proof is similar to the proof of \cite[Proposition~4.13]{CLS19}.
\end{proof}

\subsection*{Acknowledgment}  The authors would like to thank V.~Chari for some enlightening conversations and an anonymous referee for the thorough reading of our paper. L. C. was supported by Fapemig Grant (APQ-02768-21) and by CNPq Grant (402449/2021-5).


\end{document}